\documentclass[11pt]{amsart}

\usepackage{amsmath}
\numberwithin{equation}{section}
\usepackage{amsthm}
\usepackage{amssymb}
\usepackage[centering]{geometry}                
\usepackage{url}
\usepackage{verbatim}
\usepackage{graphicx}

\newcommand{\C}{\mathbf{C}}
\theoremstyle{definition}
\newtheorem{Remark}{Remark}

\begin{document}
\title[Subdiffusivity of random walk on IPC]{Subdiffusivity of random walk on the 2D invasion percolation cluster}
\date{\today}
\author{Michael Damron}\thanks{M. D. is supported by an NSF postdoctoral fellowship and NSF grant DMS-0901534 and DMS-1007626. email:~mdamron@princeton.edu}
\author{Jack Hanson}\thanks{J. H. is supported by an NSF graduate fellowship and NSF grant PHY-1104596. email:~jthanson@princeton.edu}
\author{Philippe Sosoe}\thanks{P. S. is supported by an NSERC postgraduate fellowship and NSF grant PHY-1104596. email:~psosoe@princeton.edu}
\theoremstyle{plain}
\newtheorem{theorem}{Theorem}
\newtheorem{lemma}{Lemma}

\begin{abstract}
We derive quenched subdiffusive lower bounds for the exit time $\tau(n)$ from a box of size $n$ for the simple random walk on the planar invasion percolation cluster. The first part of the paper is devoted to proving an almost sure analog of H. Kesten's subdiffusivity theorem for the random walk on the incipient infinite cluster and the invasion percolation cluster using ideas of M. Aizenman, A. Burchard and A. Pisztora. The proof combines lower bounds on the instrinsic distance in these graphs and general inequalities for reversible Markov chains. In the second part of the paper, we present a sharpening of Kesten's original argument, leading to an explicit almost sure lower bound for $\tau(n)$ in terms of percolation arm exponents. The methods give $\tau(n) \geq n^{2+\epsilon_0 + \kappa}$, where $\epsilon_0>0$ depends on the instrinsic distance and $\kappa$ can be taken to be $\frac{5}{384}$ on the hexagonal lattice.
\end{abstract}

\maketitle

\section{\large{Introduction}}
H. Kesten has proved \cite{kesten} that the simple random walk $\{X(n)\}_{n\ge 0}$ started at $\mathbf{0}$ on the incipient infinite cluster (IIC) \cite{kesten2} in two-dimensional Bernoulli bond percolation is subdiffusive in the sense that there exists $\epsilon>0$ such that the family
\begin{equation}\label{eq: tightness}
\{n^{-1/2+\epsilon}X(n)\}_{n\ge 0}
\end{equation}
is tight.
The purpose of the current work is to explain how a ``quenched'' version of this result can be obtained and extended to the random walk in an environment generated by a related two-dimensional model, invasion percolation. (The model is defined in the next section).  We present some refinements of Kesten's method, which provides a general framework for proving subdiffusivity of random walks in stochastic geometric models. In the case of two-dimensional invasion percolation (as well as the incipient infinite cluster), the ideas in \cite{kesten} can be used to give explicit bounds on $\epsilon$ from \eqref{eq: tightness} in terms of known critical exponents (see \eqref{eq: onearmprob} and \eqref{eq: twoarmprob} below).

Our main result is the following:
\begin{theorem}\label{main-theorem}
Let $\{X(k)\}_{k\ge 0}$ be a simple random walk on the invasion percolation cluster (IPC), and $\tau(n)$ the first time $X(k)$ exits the box $S(n)=[-n,n]^2$:
\[\tau(n) = \inf\{k\ge 0: |X(k)|_\infty =n\}.\] There exists $\epsilon_0>0$ such that, for almost every realization of the random walk, and almost every realization of the IPC, there is a (random) $n_0$ such that
\[\tau(n) \ge n^{2+\kappa+\epsilon_0} \text{ for } n \geq n_0.\]
$\kappa$ is a constant that can be estimated in terms of the behaviour of the one-arm and two-arm probabilities in critical percolation (with measure $\mathbb{P}_{p_c}$):
\[\kappa \ge \frac{1}{2}\eta_1 \eta_2, \]
where $\eta_1, \eta_2>0$ are exponents such that
\begin{equation}\label{eq: onearmprob}
\mathbb{P}_{p_c}(\mathbf{0} \text{ is connected to } [n,\infty)\times \mathbb{R}) \le \C_1 n^{-\eta_1}
\end{equation}
and
\begin{equation} \label{eq: twoarmprob} \mathbb{P}_{p_c}\left(\begin{gathered} \mathbf{0} \text{ has two disjoint open}\\ \text{ connections to } S(n)^c\end{gathered}\right) \le \C_2 n^{-\eta_2},\end{equation}
for some constants $\C_1$ and $\C_2$. 
\end{theorem}

\begin{Remark}If one repeats the arguments of this paper in the setting of random walk on the IIC or IPC of the hexagonal lattice, one can use the exact values of the one-arm and two-arm exponents to give a stronger bound on $\kappa$. Indeed, it is not necessary in that case to use the van den Berg-Kesten inequality \cite{vdbkesten} in \eqref{eq: missing_BK}, therefore giving
\[
\kappa \geq (1/2)\eta_2(\eta_2-\eta_1).
\]
Using the conjectured value $\eta_2 = \frac{5}{48} + \frac{1}{4}$ \cite{BN11}, we get a lower bound $\frac{17}{384}$. Without using this value, but using $\eta_2 \geq 1/4$ \cite{BN11}, we get $\kappa \geq \eta_1/8 =  \frac{5}{384}$.
\end{Remark}

\begin{Remark}This result is stronger than the corresponding theorem for the random walk on the IIC stated in \cite[Theorem~1.27]{kesten}, but it is derived by a modification of the strategy used there. In particular, Kesten proves that
\[\mathbb{P}(\tau(n)\ge n^{2+\epsilon})\rightarrow 1,\] 
for the ``averaged'' measure $\mathbb{P}$, which incorporates averaging with respect to the IIC measure constructed in \cite{kesten2}. Closer examination of his proof reveals that one can take $\epsilon = \eta^2_1/4$, and that the estimates in \cite{kesten} are sufficient to establish a ``quenched'' result by a simple application of the Borel-Cantelli lemma. A substantial part of the present paper is concerned with presenting arguments to overcome the difficulties in adapting Kesten's proof to the invasion percolation cluster.
\end{Remark}

The second result of the paper concerns a simple derivation of subdiffusivity of random walk on the IPC using results in \cite{aizburch} and \cite{pisztora} concerning the length of the shortest path from the origin to $\partial S(n)$ (the \emph{chemical distance}) in near-critical percolation. The work of these authors implies that for large $n$ this length is of order at least $n^s$, where $s>1$. Although Theorem~\ref{kesten-ipc-1} is contained in Theorem~\ref{main-theorem}, it is of interest because its proof represents a significant reduction in complexity from the original argument of Kesten.

\begin{theorem}[Quenched Kesten theorem for the IPC]
\label{kesten-ipc-1}
Let $\tau(n)$ be the time for a random walker on the invasion percolation cluster to exit $S(n)$. There exists $\epsilon > 0$ such that, for $\mathbb{P}_{\mathrm{IPC}}$-almost every $\omega$ and almost-every realization of the random walk,
\[\tau(n)\ge n^{2+\epsilon}\]
for $n$ greater than some random $n_0$.
\end{theorem}

\begin{Remark}A similar, but simpler, argument applies to the incipient infinite cluster and gives an alternative proof that the random walk on the IIC is almost surely subdiffusive. See the Appendix for details.
\end{Remark}

\begin{Remark}$\epsilon>0$ in the statement of Theorem \ref{kesten-ipc-1} depends on the value of $s$ obtained by the methods of Aizenman-Burchard and Pisztora. $s$ is both very small and difficult to calculate explicitly. Kesten's comparison argument (explained in Section~\ref{section-comparison}) yields an improvement of the estimate for $\tau(n)$ in the previous theorem by a factor of the form $n^{\kappa}$, which leads to Theorem \ref{main-theorem}. We note that any explicit bound on $s$ would be directly reflected in that theorem. Indeed, if one has upper and lower bounds (with high enough probability)
\[
Cn^{s_1} \leq \text{dist}_{\text{IPC}}(x,y) \leq Cn^{s_2},~ x,y \in \text{IPC}
\]
then one can get the lower bound $\tau(n) \geq Cn^{a}$ for any $a$ satisfying
\[
a < 2s_1 + \eta_1\left(2 \frac{s_1}{s_2} - \frac{2-\eta_2}{s_2}\right).
\]
On the hexagonal lattice, this can be improved as above to
\[
a < 2s_1 + (\eta_2-\eta_1)\left(2 \frac{s_1}{s_2} - \frac{2-\eta_2}{s_2}\right).
\]
One can actually show $s_2$ can be taken to be $2-\eta_2$, which yields the improved bound (assuming again the exact value of $\eta_2$)
\[
\kappa \geq \frac{\eta_2(\eta_2-\eta_1)}{2-\eta_2} = \frac{17}{316}.
\]
\end{Remark}

\begin{Remark}The improvement due to $s_1$ and $s_2$ in the previous remark comes from choosing $q$ larger in \eqref{eq: qdef}. It is actually a common misconception that Kesten's original ``lost in the bushes'' argument gives a lower bound for $\tau(n)$ proportional to the ratio of volume of the IIC to the volume of its backbone. The reason this is false is that it is not clear how to increase $q$ to order $n$. The parameter $q$ gives the scale at which volume estimates can be applied.
\end{Remark}

There has been little success with rigorous results for random walks on low-dimensional critical models (for instance, the IIC and IPC). One notable example is the work of D. Shiraishi \cite{S12} on random walk on non-intersecting two-sided random walk trace. For results in high dimensions, we mention the recent work of G. Kozma and A. Nachmias \cite{KN} on the IIC in dimensions $d \geq 19$ and of M. Barlow, A. J\'arai, T. Kumagai and G. Slade on the IIC for oriented percolation \cite{BJKS}. On a critical Galton-Watson tree, Kesten \cite{kesten} found the asymptotics of $\tau(n)$ and constructed a scaling limit for random walk on the IIC (see also \cite{CK} and \cite{BK}). Later, O. Angel, J. Goodman, F. den Hollander and G. Slade \cite{AGHS} found similar results for random walk on the IPC on a regular tree.

After setting some notation below, we give the definition of the invasion percolation model in Section \ref{section-invasion}, and recall some useful properties of the IPC derived in previous literature. We then prove Theorem \ref{kesten-ipc-1} in Section \ref{section-backbone}, and explain how Kesten's volume comparison argument is used to obtain Theorem \ref{main-theorem} in Section \ref{section-comparison}. Section \ref{section-volume} contains the derivation of estimates used in the proof of Theorem 1.

For convenience, we work on the square lattice $\mathbb{Z}^2$, but our results extend to planar lattices for which the Russo-Seymour-Welsh estimates hold true.

\subsection{Notation}
In this section, we give notation used throughout the paper for future reference. For any vertex (lattice point) $v =(v_1,v_2)\in \mathbb{Z}^2$, $S(n,v)$ is the box
\begin{align*}
S(n,v) &= ([v_1-n,v_1+n]\times[v_2-n,v_2+n])\cap \mathbb{Z}^2\\
&= \{x\in \mathbb{Z}^2: |x-v|_\infty \le n\}.\\
|x|_\infty &= \max (|x_1|,|x_2|).
\end{align*}
$S(n)$ is the box $S(n,\mathbf{0})$, centred at the origin. $\partial S(n,v)$ refers to the internal vertex boundary of $S(n,v)$:
\[\partial S(n,v) =\{x \in S(n) : \exists y\in S(n)^c, |x-y|_\infty = 1\}\]
We also define
\[\Lambda(n)=S(n)\cap \operatorname{IPC},\]
where IPC is defined in the next section. For a graph $G$, the set of edges is denoted by $E(G)$.

For each $p\in [0,1]$, the independent bond percolation measure $\mathbb{P}_p$ is an infinite product of Bernoulli measures with parameter $p$ indexed by the edges of $\mathbb{Z}^2$. For a finite set $I\subset E(\mathbb{Z}^2)$, and a vector $v\in \{0,1\}^I$ we have
\[\mathbb{P}_p(\sigma \in \{0,1\}^{E(\mathbb{Z}^2)}: \sigma(e) = v(e) \text{ for } e\in I) = p^{\sharp\{e: v(e) = 1\}}(1-p)^{\sharp\{e: v(e) = 0\}}.\]
A \emph{configuration} $\sigma$ is an element of $\{0,1\}^{E(\mathbb{Z}^2)}$. An edge $e$ is said to be \emph{open} in the configuration $\sigma$ if $\sigma(e)=1$, and \emph{closed} otherwise.

If $A$ and $B$ are subsets of $\mathbb{Z}^2$ we denote by
\[\mathbb{P}_p(A\rightarrow B)\]
the probability of the event that $A$ and $B$ are connected by a path of open edges. The notation $\mathbb{P}_{\mathrm{IPC}}(A\rightarrow B)$  is defined analogously.  We write
\[\mathbb{P}(A\xrightarrow{p} B)\]
to denote the probability that $A$ and $B$ are connected by $p$-open edges.

We will use the connection probabilities $\pi$ and $\rho$ defined as
\begin{align*}
\pi(p,n)&=\mathbb{P}_{p}(\mathbf{0}\to [n,\infty)\times \mathbb{R})\\
\rho(p,n)&=\mathbb{P}_{p}(\mathbf{0} \to \partial S(n) \text{ by two disjoint open paths}).
\end{align*}
These probabilities refer to independent bond percolation with parameter $p$. When no parameter is specified, it is understood that $p=p_c=1/2$; that is,
\[\pi(n)= \pi(p_c,n), \ \rho(n)= \rho(p_c,n).\]

We denote by $\mathbb{P}_{\mathrm{IPC}}$ the invasion percolation measure on bond configurations in $\mathbb{Z}^2$. Throughout, $\omega$ will denote a realization of the IPC; that is, a subgraph of $\mathbb{Z}^2$ sampled from $\mathbb{P}_{\mathrm{IPC}}$.
For each such $\omega$, we denote by $\mathbf{P}^\omega$ the probability measure associated with the simple random walk on the invasion cluster in the realization $\omega$ (which by definition contains the origin $\mathbf{0}=(0,0)\in \mathbb{Z}^2$). 

For $x>0$ , we denote by $\log x=\log_2 x$ the logarithm of $x$ in base 2.

Throughout the paper, $\C_i$ will denote constants chosen independent of $n$. We use the notation $A \lesssim B$ if there exists a constant $C$ such that 
\[A \le C B.\] This notation is only used if the implicit constant $C$ is deterministic; that is, it does not depend on the realization of the IPC or of the random walk. The notation
\[A\lesssim_c B\]
is used to emphasize that the implicit constant depends on the parameter $c$.

The notation $A \asymp B$ denotes the existence of two positive constants $D_1$ and $D_2$ such that
\[D_1B \le A \le D_2 B.\]

If $f(n)$ and $g(n)$ are two positive sequences, we use the notation $f(n) \gg g(n)$ to mean 
\[\lim_{n\rightarrow \infty} \frac{g(n)}{f(n)} = 0.\]

\section{\large{Invasion percolation}}
\label{section-invasion}
\subsection{Definition of the model}
The planar invasion percolation cluster is a random subgraph of the lattice $\mathbb{Z}^2$ which can be constructed from the familiar coupling of the independent bond percolation measures $\mathbb{P}_p$, $0<p<1$. To every edge $e$ of $\mathbb{Z}^2$, viewed as a graph, associate a random variable $w(e)$, uniformly distributed in $[0,1]$, $\{w(e) : e \in E\}$ and $\{w(e') : e' \in E'\}$ being independent for $E\cap E' = \emptyset$. An edge is called $p$-\emph{open} if $w(e)\le p$ and is $p$-\emph{closed} otherwise. The distribution of the set of $p$-open edges is that of a Bernoulli bond-percolation process at density $p$. The distribution of $(w(e))_{e\in E(\mathbb{Z}^2)}$ is a product measure which will be denoted by $\mathbb{P}$.

The IPC consists of a union of subgraphs of $\mathbb{Z}^2$ constructed by an iterative process: we start at the origin $\{0, \emptyset \}\equiv G_0$. At every stage, we form $G_{i+1}$ by adding to the current (finite) graph $G_i$ the edge $e$ with the least weight $w(e)$ among
\[\Delta G_i =\{e=(u,v)\in E(\mathbb{Z}^2), e\notin E(G_i) \text{ but } u\in G_i\}\]
as well as the endpoints of $e$. The IPC is defined to be the union $\bigcup_{i\ge 0} G_i$.

Since the percolation probability $\theta(p) = \mathbb{P}_p(\mathbf{0} \to \infty)$ at $p_c=1/2$ is zero \cite{kesten4}, the IPC contains infinitely many edges $e$ with $w(e)>1/2$. On the other hand, for any $p>p_c$, by the Russo-Seymour-Welsh theorem, the IPC will intersect the (unique) $p$-open infinite cluster almost surely (see \cite{2chayesnewman} for general $d$).  By construction, once an edge $e$ in the $p$-open infinite cluster has been added, all edges added to the IPC after $e$ have weight no bigger than $p$.

\subsection{Correlation length}
We will later require bounds for the probability that the IPC intersects the $p$-open infinite cluster, for some fixed $p$, by the time it reaches an annulus of size $n$. Such estimates can be found in \cite{zhang}, \cite{jarai}. An important notion in this context is the \emph{finite-size scaling length} $L(p,\epsilon)$. To define it, consider for $p>p_c$ the probability
\[\sigma(n,m,p)=\mathbb{P}_p(\exists \text{ a } p\text{-open horizontal crossing of } [0,n]\times[0,m]).\]
Then $L(p,\epsilon)$ is defined to be
\[L(p,\epsilon)=\min_{n\ge 0}\{\sigma(n,n,p)\ge 1-\epsilon\}.\]
From \cite{kesten3}, it is known that $L(p,\epsilon_1)\asymp L(p,\epsilon_2)$ for $0<\epsilon_1,\epsilon_2 \le \epsilon_0$, so we shall fix $\epsilon =\epsilon_0$ and henceforth simply refer to $L(p)=L(p,\epsilon)$. We note the following properties of $L$:
\begin{enumerate}

\item $L(p)$ is right-continuous, non-increasing in $(p_c,1)$ and $L(p)\rightarrow \infty$ as $p\downarrow p_c$.
\item Taking $\epsilon_0$ small enough, there exists $K>0$ such that \cite[(2.8)]{jarai}:
\[\sigma(2mL(p),mL(p),p)\ge 1 - \exp(-Km),~ m \geq 1.\]
\item Again from \cite[Eq.~2.10]{jarai}, there exists $D$ independent of $p$ such that
\[\lim_{\delta\downarrow 0} \frac{L(p-\delta)}{L(p)} \le D.\]

\end{enumerate}

Let $\log^{(j)}$ be the $j$-th iterate of $\log$, and
\[\log^*n =\min \{j>0 : \log^{(j)}n \le 16\}, \ n\ge 16.\]
Define, for $n\ge 16$ and $j=1,\ldots,\log^* n$
\[p_n(j) = \min_{p>p_c} \left\{ L(p) \le \frac{n}{M\log^{(j)}n}\right\}.\]
$M>0$ is a constant to be determined later.
Note that if $m\le n$, then $p_n(1)\ge p_m(1)$ when $m$ is sufficiently large.

By (3) above, there exists $D>0$ such that
\begin{equation}
\label{eq:l_p}
M\log^{(j)}n \le \frac{n}{L(p_n(j))}\le DM\log^{(j)}n. \end{equation}

Item (2) in the list above implies \cite[(2.21)]{jarai}
\begin{align}
\label{eq:jarai-bound}
\mathbb{P}(H_n(j)) &\equiv \mathbb{P}\left(\begin{gathered} \exists\  p_n(j)\text{-open circuit } \mathcal{C} \text{ around } 0 \text{ in } S(n/2)\setminus S(n/4) \\\text{ and }\mathcal{C} \text{ is connected to } \infty \text{ by a }p_n(j)\text{-open path}\end{gathered} \right)\\
& \ge 1- \C_3 \exp(-\C_0M\log^{(j)}n).\nonumber
\end{align}
The measure $\mathbb{P}$ refers to the coupling of the $p$-Bernoulli measures described earlier. If the event $H_n(j)$ occurs, the IPC intersects the $p_n(j)$-open infinite cluster by the time it reaches $S(n)^c$. The bound (\ref{eq:jarai-bound}) plays a role in estimates derived in Section \ref{section-volume}.

\section{\large{Proof of Theorem~\ref{kesten-ipc-1}}}
\label{section-backbone}

We begin by giving a brief sketch of the main idea. The first step is to consider a restriction of the random walk to a certain subset of the IPC, the backbone. The exit time for this walk from a box of size $n$ is controlled using the Varopolous-Carne inequality. This inequality implies that the exit time is at least of order $d^2$, where $d$ is the chemical (instrinsic) distance to the boundary of the box of size $n$ through the IPC. In Lemma~\ref{pisztoraslemma}, we outline an argument of A. Pisztora that proves that $d$ grows superlinearly with $n$. All of these estimates are tight enough to apply Borel-Cantelli and close the proof of subdiffusivity. 

\subsection{Random walk on the backbone}
The simple random walk started at $\mathbf{0}$ on the IPC is the Markov chain $\{X(k)\}_{k\ge 0}$ with the set of sites in the IPC as its state space, such that $X(0) =\mathbf{0}$, and with transition probabilities given by
\begin{equation*}
\mathbf{P}^\omega(X(k+1) =y \mid X(k)=x) = \frac{\mathbf{1}[(x,y) \in E(\mathrm{IPC})]}{\operatorname{deg}(x,\mathrm{IPC})}.
\end{equation*}
The random variable $\operatorname{deg}(x,\mathrm{IPC})$ denotes the number of sites $y$ such that the edge $(x,y)$ belongs to the $\mathrm{IPC}$.

Below, it will be convenient to work with a modification of $X$ that is reversible on $\Lambda(n)$. Thus, we let $\{X^n(k)\}_{k\ge 0}$ be the Markov chain started at the origin and defined by the transition probabilities
\[\mathbf{P}^\omega(X^n(k+1) =y \mid X^n(k)=x) = \frac{\mathbf{1}[(x,y) \in E(\Lambda(n))]}{\operatorname{deg}(x,\Lambda(n))}.\]
Note that the distribution of $X^n(k)$ coincides with that of $X(k)$ for $k\le \mathcal{\tau^*}(n)$, where
\[\tau^*(n) = \inf \{k\ge 0: |X^n(k)|_\infty=n\}.\]
Moreover, the distribution of $\tau^*(n)$ is equal to the distribution of the exit time $\tau(n)=\inf\{k \geq 0:|X(k)|_\infty =n\}$ defined in terms of the ``full'' random walk $X$ on $\Lambda(n)$. Thus, it will suffice to obtain bounds on $\tau^*(n)$.

The ``backbone'' $B(n)$ of $\Lambda(n)$ is the set of sites in $\Lambda(n)$ connected in the invasion cluster to $\mathbf{0}$ and to $\partial S(n)$ by two disjoint paths. A simple argument (see \cite[Lemma~3.13]{kesten}) shows that whenever $X^n$ leaves the backbone, it must return at the site where it left before it reaches $\partial{S}(n)$. Thus the random walk $X^n$ on $\Lambda(n)$ induces a random walk $X^{n,B}$ on $B(n)$ which moves only when $X^n$ is in $B(n)$. That is, if we define
\begin{align*}
\sigma_0 &= 0\\
\sigma_m &= \inf \{k > \sigma_{m-1}: X^n(k) \in B(n)\}\\
X^{n,B}(k)&\equiv X^n(\sigma_k), 
\end{align*}
then $X^{n,B}$ is a random walk on the backbone $B(n)$, with transition probabilities given by
\begin{equation}
\mathbf{P}^\omega(X^{n,B}(k+1)=y\mid \,X^{n,B}(k)=x)=\begin{cases}
\frac{\mathbf{1}[x,y\in B(n), (x,y)\in E(\mathbb{Z}^2))]}{\operatorname{deg}(x,\Lambda(n))},&  y\neq x\\
\frac{\operatorname{deg}(x,\Lambda(n)) -\operatorname{deg}(x,B(n))}{\operatorname{deg}(x,\Lambda(n))}, & x=y.
\end{cases}
\end{equation}
Here, $\text{deg}(x,B(n))$ is defined as the number of edges $(x,y)$ in $\Lambda(n)$ such that $x,y \in B(n)$.

\subsection{Estimate on the speed of the walk}

Irrespective of the geometry of  $B(n)$, $X^{n,B}$ must travel at least $n$ steps in $B(n)$ to reach $\partial{S}(n)$, because the distance between any two points in $B(n)\subset \mathbb{Z}^2$ is no less than the corresponding chemical distance in $\mathbb{Z}^2$. This fact was used by Kesten to conclude that the time spent by the walker on the backbone is of order at least $n^2/\log n$ with high probability.  The Carne-Varopoulos bound (\cite{carne}, \cite{varopoulos}; see also \cite[Theorem~13.4]{pereslyons}) allows us to obtain a better estimate by considering the chemical distance on $B(n)$. It implies that the reversible Markov chain $X^{n,B}$ has at most diffusive speed in the intrinsic metric of the backbone. If $\mu$ is the stationary measure for the walk $X^{n,B}$ ($\mu$ depends on $\omega$), then
\[\mathbf{P}^\omega(X^{n,B}(k)=y\,\mid X^{n,B}(0)=\mathbf{0})\le 2\sqrt{\mu(y)/\mu(0)} \exp(-\operatorname{dist}_{B(n)}(\mathbf{0},y)^2/(2k)).\]
The right side of this expression refers to the chemical distance in the backbone $B(n)$. The ratio appearing on the right can be bounded independently of the realization $\omega$ of the invasion percolation, since the stationary measure $\mu$ satisfies
\[ 1/4 \le\frac{\mu(x)}{\mu(y)}\le 4,\]
for any $x,y\in B(n)$.
Since $B(n)\subset \mathbb{Z}^2$, we have the inequality of graph distances:
\[\operatorname{dist}_{B(n)} \ge \operatorname{dist}_{\mathbb{Z}^2} = \operatorname{dist}_1.\]
Summing this bound over $\lambda \sqrt k\le|y|_\infty \le k$, we find
\begin{align*}
\mathbf{P}^\omega(\operatorname{dist}_{B(n)}(\mathbf{0},X^{n,B}(k))\ge \lambda\sqrt{k} )\lesssim k^2\exp(-\lambda^2/\C_4).
\end{align*}
Suppose we restrict our attention to realizations $\omega$ of the environment such that the chemical distance in $B(n)$ satisfies
\[\operatorname{dist}_{B(n)}(\mathbf{0},\partial S(k)) \ge \C_5 k^s, \quad  k \ge n_0(\omega)\]
for  some $n_0(\omega)$ and some deterministic constants $s>1$, $\C_5>0$. For such $\omega$, $\lambda\ge 1$ and $n$ sufficiently large, we have:
\begin{align}
\label{eq-strong-cv}
\mathbf{P}^\omega(|X^{n,B}(k)|_\infty \ge \lambda k^{1/(2s)} ) &\leq \mathbf{P}^\omega(\operatorname{dist}_{B(n)}(X^{n,B}(k),\mathbf{0}) \ge \C_5 \lambda k^{1/2}) \\
&\lesssim k^2\exp(-\lambda^2/\C_6). \nonumber
\end{align}

\subsection{Chemical distance in the IPC}
It follows from work of Aizenman and Burchard \cite{aizburch} that the chemical distance inside a large box in independent bond percolation with parameter $p_c=1/2$ is bounded below by a power $s>1$ of the Euclidean distance in $\mathbb{Z}^2$ with high probability. Pisztora \cite{pisztora} showed how to extend this result to $p>p_c$ suitably close to $1/2$, and to the invasion percolation cluster. We reproduce the argument leading to his result, in a form that suits our needs, in the lemma below. Theorem~\ref{kesten-ipc-1} follows from these results and the considerations above.

\begin{lemma}[\cite{pisztora},~Theorem~1.3]
\label{pisztoraslemma}
There exist $\C_7$ and $s>1$ such that
\begin{equation}\label{eq:invasion-chemical}
\mathbb{P}_{\mathrm{IPC}}\left(\operatorname{dist}_{\Lambda(n)}(\mathbf{0},\partial S(n))\le \C_7 n^{s} \right) \lesssim n^{-2}.
\end{equation}
\begin{proof}
The models considered in \cite{aizburch} are defined by families $\{\mathbf{P}_\ell\}_{\ell>0}$ of probability measures on collections of curves in a compact region $\mathcal{R}$. For each $\ell$, $\mathbf{P}_\ell$ is supported on unions of polygonal curves with step size $\ell$. The realizations in the support of $\mathbf{P}_\ell$ are denoted by $\mathcal{F}_\ell$.

A truncated version of capacity is used to obtain lower bounds on the minimal number $N(A, \ell)$ of sets of diameter $\ell$ required to cover a given set $A\subset \mathcal{R}$:
\begin{equation}
\label{eq:lowerbd}
N(A, \ell) \ge \operatorname{cap}_{s,\ell}(A)\cdot \ell^{-s}
\end{equation}
where 
\[\operatorname{cap}^{-1}_{s,\ell}(A) = \inf_{\mu(A)=1}\iint \frac{1}{\max(|x-y|^s , \ell^{-s})}\ \mu(\mathrm{d}x)\mu(\mathrm{d}y).\]
The infimum is over Borel probability measures supported on $A$.

Under the assumption  ``Hypothesis H2,'' the authors of \cite{aizburch} obtain uniform bounds for $\operatorname{cap}_{s,\ell}$: if there exist some $K, \sigma >0$,  and $0<\rho<1$ such that for every $k$ and collection of $k$ rectangles $A_1,\ldots, A_k$ of lengths $l_1,\ldots,l_k \ge \ell$ and cross-section\footnote{The cross-section of a rectangle is the ratio of its short side to its long side.} $\sigma$, and satisfying
\[\operatorname{dist}(A_j,\cup_{i\neq j} A_i) \ge \operatorname{diam}A_j\]
for all $j$, we have
\[\mathbf{P}_\ell(\text{all } A_i \text{ are traversed by segments of a curve in } \mathcal{F}_\ell)\le K\rho^k,\]
then the capacity $\operatorname{cap}_{s,\ell}$ of macroscopic curves is bounded below for some $s>1$ \cite[Theorem~1.3]{aizburch}: all curves $\mathcal{C}$ in $\mathcal{F}_\ell$ with $\operatorname{diam}(\mathcal{C})\ge 1/10$ satisfy
\[\operatorname{cap}_{s,\ell}(\mathcal{C})\ge C(s,\omega,\ell).\]
$C(s,\omega,\ell)$ is a random variable which is \emph{stochastically bounded below} in the sense that 
\begin{equation}\label{eq:straight-runs}
\mathbf{P}_\ell(C(s,\omega,\ell)\le u)\rightarrow 0
\end{equation}
uniformly in $\ell$ as $u\rightarrow 0$.

We will apply the results in \cite{aizburch}, with $\ell = n^{-1}$ to bond percolation on the rescaled lattice 
\[\mathcal{R}_n = (1/n)\mathbb{Z}^2\cap [-1,1]^2.\] 
For $p\in [0,1]$, let $\mathbb{P}^n_p$ denote the independent bond percolation measure with parameter $p$ on the edges of $\mathcal{R}_n$. $\mathbb{P}^n_p$ induces a probability measure on configurations $\mathcal{F}_{1/n}$ of curves in $\mathcal{R} =[-1,1]^2$: the percolation configuration is a union of connected paths of $p$-open edges, each edge being identified with a line segment of length $1/n$.

In the case of independent percolation, Hypothesis H2 reduces to the existence of a cross-section $\sigma$ and $\rho<1$ such that the probability that there exists an open-crossing of a rectangle of cross-section $\sigma$ is less than $\rho$. By the Russo-Seymour-Welsh estimates, Hypothesis H2 is satisfied for $\{\mathbb{P}^n_{p_c}\}_{n\ge 1}$. 

The lower bound (\ref{eq:lowerbd}) gives an estimate for the chemical distance in $\mathcal{F}_{1/n}$ between any two sets in $[-1,1]^2$. Any $p_c$-open path in $\mathcal{R}_n$ connecting subsets $A$ and $B$ of $\mathcal{R}_n$ at Euclidean distance
\[\operatorname{dist}(A,B)\ge 1/10\]
contains at least $C(s,\omega,1/n) \cdot n^s$ bonds. Denote by $\operatorname{dist}_{\mathcal{F}_{1/n}}(A,B)$ the (random) number of bonds in the shortest $p_c$-open path connecting $A$ and $B$ in $\mathcal{R}_n$.  By (\ref{eq:straight-runs}), given any $\epsilon>0$, we can choose $C(\epsilon)$ such that for all $n$,
\[\mathbb{P}^n_{p_c}(\operatorname{dist}_{\mathcal{F}_{1/n}}(A,B) \le C(\epsilon)\cdot n^s)\le \epsilon.\]
The scaling 
\[x\mapsto nx\] defines a measure-preserving bijection between $(E(\mathcal{R}_n),\mathbb{P}^n_{p_c})$ and $(E(S(n)), \mathbb{P}_{p_c})$. It follows that for each $\epsilon>0$, there exists a constant $C(\epsilon)$ such that for all subsets $A$, $B$ of $S(n)$ at Euclidean distance $n/10$ from each other, 
\[\mathbb{P}_{p_c} \left(\begin{gathered} \text{there exists an open path connecting } A \text{ and } B \\\text{ in } S(n)
\text{ with no more than } C(\epsilon) n^s \text{ bonds} \end{gathered}\right)\le \epsilon.\]
Note that if $B$ cuts $A$ from $S(n)^c$ in $\mathbb{Z}^2$, the restriction that the path be contained in $S(n)$ is superfluous. This point will be relevant below.

The observation in \cite{pisztora} is that the Aizenman-Burchard bounds remain valid for $p>p_c$ as long as $n$ is smaller than the correlation length $L(p)$. The estimate used to obtain (\ref{eq:straight-runs}) depends only on $\sigma$ and $\rho$ \cite[p.~446]{aizburch}. It follows from the definition of $L(p)$ and the Russo-Seymour-Welsh estimates that there exists $\rho<1$ such that  for rectangles of cross-section ratio $1/3$, say, with long side $n\le 3L(p)$,
\[\mathbb{P}_p(\exists \text{ an open crossing of } [0,n]\times[0,n/3])\le \rho.\]

Thus (\ref{eq:straight-runs}) remains true uniformly for $\ell^{-1}\le 3L(p)$. Repeating the argument above, we see that we can choose $C(\epsilon)$ independent of $p\in(p_c,1)$ to make the probability
\[\mathbb{P}_{p}(\operatorname{dist}(\partial S(L(p)),\partial S(3L(p)) \le C(\epsilon) L(p)^s)\]  
smaller than an arbitrary $\epsilon>0$.
The distance refers to the chemical distance in the union of all percolation clusters in the box $S(3L(p))$. Since $L(p)\rightarrow \infty$ as $p\downarrow p_c$, for any fixed $\epsilon$, $L(p)$ is much greater than $C(\epsilon)$, and so the estimate on the distance is not vacuous. More precisely, we find
\begin{equation}
\label{eq:c-epsilon}
\limsup_{p\downarrow p_c}\mathbb{P}_{p}(\operatorname{dist}(\partial S(L(p)),\partial S(3L(p)) \le C(\epsilon) L(p)^s) \le \epsilon.
\end{equation}

A block argument with blocks of size $3L(p)$ converts the initial estimate (\ref{eq:c-epsilon}) into an exponential bound for the macroscopic chemical distance in near-critical percolation (see the proof of \cite[Theorem~1.3,~pp.~12-14]{pisztora}). There exist constants $\C_8$, $\C_9$ such that if $p$ is sufficiently close to $p_c$:
\begin{equation}
\label{eq:renormalized-ab}
\mathbb{P}_p(\operatorname{dist}(\mathbf{0},\partial S(n))\le \C_8 n(L(p))^{s-1} )\lesssim \exp\left(-\C_9 \frac{n}{L(p)}\right).
\end{equation}
With this in hand, (\ref{eq:invasion-chemical}) follows from the construction described in Section \ref{section-invasion}. We outline the argument. The occurrence of the event 
\[\tilde H_n(1) = \left\{ \begin{gathered} \exists\  p_{(n/4)-1}(1)\text{-open circuit } \mathcal{C} \text{ around } \mathbf{0} \text{ in } S(n/2)\setminus S(n/4)\\ \text{ and }\mathcal{C} \text{ is connected to } \infty \text{ by a }p_{(n/4)-1}(1)\text{-open path}\end{gathered}\right\}\] implies that all edges of the IPC in $\Lambda(n) \setminus S(n/2)$ are $p_{(n/4)-1}(1)$-open.
\begin{align}
\label{eq: finaldistance}
&\mathbb{P}_{\mathrm{IPC}}(\operatorname{dist}_{\Lambda(n)}(\mathbf{0},\partial S(n))
\le (\C_8/5) \cdot n(L(p_{(n/4)-1}(1)))^{s-1})\\
\le& \ \mathbb{P}\left(\begin{gathered} \exists\  x \in \partial S(3n/4):  \operatorname{dist}_{\Lambda(n)}(x,\partial S(x,n/4-1)) \\ \le (\C_8/5)\cdot n(L(p_n(1))  )^{s-1};  \tilde H_n(1) \end{gathered}\right) +\mathbb{P}(\tilde H_n(1)^c) \nonumber
\\ 
\lesssim& \  n\mathbb{P}_{p_{(n/4)-1}(1)}(\operatorname{dist}(\mathbf{0},\partial S(n/4-1))\le (\C_8/5) \cdot n(L(p_{(n/4)-1}(1)))^{s-1} ) +n^{-M\C_{10}} \nonumber \\
\lesssim& \  n\exp\left(-\C_{11} \frac{n}{L(p_{(n/4)-1}(1))}\right) + n^{-M\C_{10}}. \nonumber
\end{align}
The final inequality follows from (\ref{eq:renormalized-ab}) and (\ref{eq:jarai-bound}). Recalling (\ref{eq:l_p}):
\[\frac{n}{L(p_n(1))}\ge M\log n,\]
and choosing $M$ suitably large in the definition of $p_n(1)$, we find:
\[ n\exp\left(-\C_{11} \frac{n}{L(p_{(n/4)-1}(1))}\right) + n^{-M\C_{10}} \lesssim n^{-2}.\]
By slightly lowering $s$ to absorb the logarithm, the probability on the left of (\ref{eq: finaldistance}) can be made to match the form of the left side of (\ref{eq:invasion-chemical}).
\end{proof}
\end{lemma}

\begin{Remark} The final part of the proof of Lemma \ref{pisztoraslemma} shows that for any $0< R_1 < R_2 < R_3$ and any $k>0$, one can find constants (depending on $R_i$ and $k$) such that
\[\mathbb{P}_{\mathrm{IPC}}(\operatorname{dist}_{\Lambda(R_3n)}(\partial S(R_2 n), \partial S(R_1 n) \cup \partial S(R_3 n)) \le Cn^s) \lesssim n^{-k}.\]
Here $s>1$ is the constant appearing in (\ref{eq:invasion-chemical}). Such a statement will be used in Section \ref{section-comparison} below.
\end{Remark}

\begin{proof}[Proof of Theorem \ref{kesten-ipc-1}]
For $s>1$, let $L_{n}$ be the event
\[\{\operatorname{dist}_{\Lambda(n)}(\mathbf{0},\partial S( n))\le \C_7 n^s\}.\]
By (\ref{eq:invasion-chemical}) in the previous lemma, we have
\begin{equation}
\label{eq:l_n-sum}
\sum_{n\ge 1} \mathbb{P}_{\mathrm{IPC}}(L_n) <\infty
\end{equation}
for some $s>1$. Applying the Borel-Cantelli lemma and choosing \[\omega \in \{L_n \text{ occurs infinitely often}\}^c,\] we can use (\ref{eq-strong-cv}) with $\lambda = 4^s \cdot \C_6^{s/2} (\log n)^{s/2}$ and $\C_5=\C_7$; for some $N(\omega)$ we have: 
\[\sum_{n\ge N(\omega)} \mathbf{P}^\omega( \tau^*(n) \le  n^{2s}/\lambda^2 ) \le \sum_{n\ge N(\omega)} \sum_{k\le n^{2s}/\lambda^2 } \mathbf{P}^\omega (|X^{n,B}(k)|_\infty\ge n)<\infty.\]
A second application of the Borel-Cantelli lemma leads to Theorem \ref{kesten-ipc-1}.
\end{proof}
Note that for the argument above it was not necessary to consider $X^{n,B}$. However, the decomposition of the IPC into a backbone and ``dangling ends'' will be central in the derivation of Theorem \ref{main-theorem} below. The proof of Theorem \ref{kesten-ipc-1} shows that $X^{n,B}$ alone already contributes at least $n^{2+\epsilon}$ steps to $\tau(n)$.

\section{\large{Kesten's comparison argument}}
\label{section-comparison}

Our modification of Kesten's argument compares the volume of sites in the invasion percolation cluster (IPC) to the volume of sites on the backbone to conclude that the walk must be subdiffusive.

\subsection{Preliminaries and a key lemma}

We assume for simplicity of notation that $n = 3m$, $m\in\mathbb{Z}^+$. We introduce two stopping times: 
\begin{align*}
\tau(2m) &= \inf\{k\ge 0: X(k) \in \partial S(2m)\}\\
\sigma^+(m) &= \inf\{ k\ge \tau(2m): X(k) \in \partial S(m) \cup \partial S(n)\}.
\end{align*}
By definition, we clearly have:
\[\tau(n)=\tau(3m) \ge \sigma^+(m)-\tau(2m).\]
Hence, it will suffice to obtain a lower bound on the right side of the previous expression.
\[Y(k) = X(\tau(2m) +k), \ k\geq 0\] is a simple random walk on IPC; now define $Y^n$ to be the simple random walk on (the possibly disconnected) 
\[
\Gamma(n) = \text{IPC} \cap (S(n) \setminus S(m))
\]
with initial point $Y(0)$. Letting $\sigma^*(n)$ be the hitting time of $\partial S(m) \cup \partial S(n)$ by the walk $Y^n$, we note that $\sigma^*(n)$ has the same distribution as $\sigma^+(m)-\tau(n)$.

A key tool in Kesten's argument is the following result from \cite{kesten}, expressing the spatial ``smoothness'' of the local times for a reversible Markov chain.

\begin{lemma}[\cite{kesten}, Lemma 3.18]\label{continuity}
Let x, y be two sites in $\Gamma(n)$, and let
\[L(x,k) =\sharp \{l:0\le l \le k,\  Y^n(l)=x\}\]
be the local time at a site $x$ of the walk $Y^n$.
Then, for some $L_0>0$ and any $\lambda>1$:
\begin{multline}
\label{eq: localtimes}
\mathbf{P}^\omega\left(\exists k, \ L(y,k)\ge \lambda\operatorname{dist}_{\Gamma(n)}(x,y) \text{ and } L(x,k)\le \frac{1}{2}\frac{\operatorname{deg}_{\Gamma(n)}(x)}{\operatorname{deg}_{\Gamma(n)}(y)}L(y,k)\right)\\ 
\lesssim \operatorname{dist}_{\Gamma(n)}(x,y)\exp(-\lambda/L_0).
\end{multline}
\end{lemma}
In \cite{kesten}, Lemma \ref{continuity} is stated in terms of the intrinsic distance on the incipient infinite cluster. Replacing $\|x-y\|_{m,w}$, $d(x)$ and $d(y)$ in  the proof of Lemma 3.18 in \cite{kesten} by $\operatorname{dist}_{\Gamma(n)}$, $\deg_{\Gamma(n)}(x)$ and $\deg_{\Gamma(n)}(y)$, respectively, we obtain Lemma \ref{continuity} above.

We also modify our definition of the backbone. $\tilde{B}(n)$ is defined to be the set of sites in $\Gamma(n)$ connected by two disjoint paths (in $\Gamma(n)$) to $\partial S(n)$ and $\partial S(m)$. $Y^{n,\tilde{B}}$ is the induced walk on $\tilde{B}$, defined analogously to $X^{n,B}$ in Section \ref{section-backbone}. We let $b(n)$ be the number of steps $Y^{n,\tilde{B}}$ takes between $0$ and $\sigma^*(n)$; $b(n)$ is the time spent by $Y^n$ on $\tilde{B}(n)$.

\subsection{Sketch of the proof of Theorem~\ref{main-theorem}}
Kesten's comparison argument will be applied to $Y^n$. The idea is to consider a ``thickening'' of the backbone of size $q$. By Lemma \ref{continuity}, if a box $S(v,q)$ of size $q$ contains a site $x\in \tilde{B}(n)$ with $L(x,\sigma^*(n))\gg q^2 L_0$, the random walk visits all accessible sites of $\Gamma(n)$ inside $S(v,q)$ at least $CL(x,\sigma^*(n))$ times, with high probability. If it is traversed by a portion of the random walk, the box $S(v,q)$ typically contains  $q^2\pi(q)$ sites of $\Gamma(n)$, and at most $q^2 \rho(q)$ sites of $\tilde{B}(n)$. Thus the time spent by $Y^n$ in $S(q,v)$ up to $\sigma^*(n)$ is larger than the time $Y^{n,\tilde{B}}$ spends there by a factor of at least $\pi(q)/\rho(q)$. By choosing $q$ appropriately, the set of sites $y$ on the backbone which do not satisfy the lower bound of order $q^2L_0$ on $L(y,\sigma^*(n))$ will make a contribution bounded by a fraction of the total time spent on the backbone.

\subsection{Proof of Theorem~\ref{main-theorem}}

To realize the strategy just described, we tile $S(n)\setminus S(m)$ by squares of size
\begin{equation}
\label{eq: qdef}
q=   Q\cdot \frac{n^{\eta_2/2}}{(\log n)^{3/2}}\end{equation}
for a constant $Q$ to be determined. Here $\eta_2$ is the exponent appearing in (\ref{eq: twoarmprob}). We note for future reference that
\begin{equation}
\label{eq: twoarmbound}
\eta_2 \leq 1 \text{ so that } q = o(\sqrt n)\ .
\end{equation}
This bound on $\eta_2$ can be proved using the method of \cite[Cor.~3.15]{vdbkesten}. For the details, the reader can see a standard sketch of a similar inequality (for crossings of an annulus) under equation \eqref{eq:twoarms} later in the paper.

For $\mathbf{j}=(j_1,j_2)\in \mathbb{Z}^2$, define
\begin{align*}
D(\mathbf{j},q) &= [q j_1, q(j_1+1))\times[q j_2,q (j_2+1)),\\
F(\mathbf{j},q) &= [q (j_1-1),q(j_1+2)]\times [q (j_2-1), q(j_2+2)].
\end{align*}
Given a realization $\omega$ of IPC and a realization of the walk, we follow the path of $Y^{n}$ until $\sigma^*(n)$ by introducing two sequences $\{l_i\}$ and $\{\mathbf{j}_i\},$ first setting $l_0=0$ and $\mathbf{j}_0$ to be the index such that $Y^{n}(l_0)\in D(\mathbf{j}_0,q)$ and then defining $l_i$ by
\begin{align*}
l_{i+1} &=\min \{l>l_i, Y(l) \notin F(\mathbf{j}_i,q)\}\\
Y^n(l_{i+1})&\in D(\mathbf{j}_{i+1},q).
\end{align*}
$Y^n$ may reach $\partial S(m) \cup \partial S(n)$ before leaving $F(\mathbf{j}_i,q)$, in which case $l_i$ ends the sequence. We let $\mathcal{C}(i)$ denote the component of $F(\mathbf{j}_i,q)\cap \Gamma(n)$ containing $Y^n(l_i)$. $Y^n$ may return several times to the same square, so $\mathcal{C}(j)$ may be equal to $\mathcal{C}(i)$ for $i\neq j$. Enumerating the $\mathcal{C}(i)$ without repetition as $\iota_0, \iota_1, \ldots , \iota_\lambda$, with $\mathcal{C}(\iota_\lambda)$ the component of $F(\mathbf{j}_i,q)$ where $i$ is such that
\[Y^n(\sigma^*(n))\in F(\mathbf{j}_i,q),\]
we define
\begin{align*} \Lambda(\iota) &= \sum_{x\in \mathcal{C}(\iota)} L(x,\sigma^*(n))\\
\Theta(\iota) &= \sum_{x\in \tilde{B}(n) \cap\mathcal{C}(\iota)} L(x, \sigma^*(n)).
\end{align*}

Since any $x$ belongs to at most $16$ different $F$ squares, we have:
\[\frac{1+\sigma^*(n)}{1+b(n)}=\frac{\sum_x L(x,\sigma^*(n))}{ \sum_{x\in \tilde{B}(n)} L(x,\sigma^*(n))}\ge \frac{1}{16}\frac{\sum_\iota \Lambda(\iota)}{\sum_\iota \Theta(\iota)}.\]

We now state volume estimates analogous to those obtained in \cite{kesten} for the incipient infinite cluster; they will be derived in the next section. We will only be concerned with those indices in the set
\[\mathcal{J}=\{\mathbf{j}\in \mathbb{Z}^2: F(\mathbf{j},q)\cap (S(n)\setminus S(m)^\circ) \neq \emptyset\}.\] 
The first estimate is for the number of backbone sites in any $F$ square; for any $\mathbf{j}\in\mathcal{J}$, we have:
\begin{equation}
\label{eq-backbone-bound} 
\mathbb{P}_{\mathrm{IPC}}\left( \sharp (\tilde{B}(n) \cap F(\mathbf{j},q))\ge  \frac{c}{\C_{12}} q^2\rho(q)\log q \right)\lesssim q^{-c}.
\end{equation}

The second provides, with high probability, a lower bound for the number of sites of the IPC in a box $F(\mathbf{j},q)$, $\mathbf{j}\in \mathcal{J}$, given that there is a crossing of $F(\mathbf{j},q)\setminus D(\mathbf{j},q)$:
\begin{equation}
\label{eq-crossing-bound} 
\mathbb{P}_{\mathrm{IPC}}\left(\begin{gathered}\text{there exists a crossing } r \subseteq \Gamma(n) \text{ of } F(\mathbf{j},q)\setminus D(\mathbf{j},q)\\ \text{ with } \sharp \{x\in F(\mathbf{j},q): x \text{ connected to } r \text{ in }\\ \Gamma(n)\cap F(\mathbf{j},q) \} \le q^{2}\pi(q)/(\log q)^4 \end{gathered}\right)\lesssim_c q^{-c}
\end{equation}
Here $c$ is arbitrary but the implicit constant depends on the choice of $c$.

\subsection{The events $E_i(n)$ and $W_i(n)$}

We now define the events $E_i(n)$, $1\le i\le 4$ and $W_i$, $1 \le i \le 3$. The ratio $\sum \Lambda(\iota)/\sum \Theta(\iota)$ will be bounded below by $\pi(q)/\rho(q)$ on the event $(\cap_i E_i) \cap (\cap_i W_i)$.

\begin{enumerate}
\item \[E_1(n) = \left\{ \omega: \operatorname{dist}_{\Gamma(n)}(\partial S(2m),\partial S(n)\cup \partial S(m)) \ge \C_{13} n^{s} \right\}.\]
\item \[E_2(n) = \left\{ \begin{gathered} \omega: \sharp (\tilde{B}(n) \cap F(\mathbf{j},q)) \le \C_{14} q^2\rho(q)\log q\\  \text{for all } \mathbf{j} \in \mathcal{J}\end{gathered} \right\}.\]

\item \[E_3(n) = \left\{ \begin{gathered}\omega: \sharp \{x\in F(\mathbf{j},q): x \text{ connected to } r \text{ in } \\
\Gamma(n)\cap F(\mathbf{j},q) \} \geq q^{2}\pi(q)/(\log q)^4 \text{ for all }\\ \mathbf{j}\in \mathcal{J} \text{ and any crossing } r \subseteq \Gamma(n) \text{ of } F(\mathbf{j},q)\setminus D(\mathbf{j},q) \end{gathered} \right\}. \]
\item \[E_4(n) = \left\{ \omega: \sharp \tilde{B}(n) \le  \frac{4}{\C_{15}} n^2\rho(n) (\log n)^2 \right\}. \]
\item \label{item-W1} \[W_1(n) = \{b(n) \geq n^{2s'}/\log n\}.\]
\item \[
W_2(n) =  \left\{  \begin{gathered}
1/8 \le \frac{L(x,\sigma^*(n))}{L(y,\sigma^*(n))} \le 8  \text{ for each pair } x,y \in S(n)\smallsetminus S^{\circ}(m)\\
 \text{ such that } x,y \text{ belong to the union of two}
 \text{ clusters }\\ \mathcal{C}(i), \mathcal{C}(i+1)
\text{ traversed consecutively by } Y^n \text{ and }\\
 L(x,\sigma^*(n)) \ge 320 L_0q^2\log n \end{gathered} \right\}.\]
 \item \[W_3(n) =\left\{ \begin{gathered} L(x,\sigma^*(n)) \le 2560 L_0q^2\log n \\ \text{ for any } x \text{ in a cluster } \mathcal{C}(i)
 \text{ such that } \\ L(y,\sigma^*(n)) \le 320 L_0 q^2\log n \\ \text{ for some } y\in \mathcal{C}(i) \end{gathered}
 \right\}.\]
\end{enumerate}

By the remark following the proof of Lemma \ref{pisztoraslemma}, there exists $\C_{13}$ such that
\[\mathbb{P}_{\mathrm{IPC}}(E_1(n)^c) \lesssim n^{-2},\]
with the same constant $s$ as in (\ref{eq:invasion-chemical}).
For any $1<s'<s$ and $\omega \in E_1(n)$, we use the Carne-Varopoulos estimate (\ref{eq-strong-cv}), applied to the symmetric chain $Y^{n,\tilde{B}}$, to show as in the proof of Theorem \ref{kesten-ipc-1}:
\begin{equation}
\label{eq: W1bound}
\mathbf{P}^\omega (W_1(n)^c) \lesssim n^{-2}, \quad \omega \in E_1(n),
\end{equation}
giving the bound (\ref{kesten-ipc-1}).

Recall the definition of $q$ in \eqref{eq: qdef}. We have 
\[q^{-1} = o(n^{-\eta_2/4}).\] Noting that there are, up to a constant, at most $n^2/q^2$ indices $\mathbf{j}$ in $\mathcal{J}$, and choosing $\C_{14}\ge 16/(\C_{12}\eta_2)$ in (\ref{eq-backbone-bound}), and accordingly in the definition of $E_2(n)$, we find
\[ \mathbb{P}_{\mathrm{IPC}}(E_2(n)^c)\lesssim n^{-2}.\]
By the estimate (\ref{eq-big-backbone-bound}) in Section \ref{section-volume}:
\[ \mathbb{P}_{\mathrm{IPC}}(E_4(n)^c) \lesssim n^{-2}.\]
By (\ref{eq-crossing-bound}) (for some $c$ large enough), we have
\[\mathbb{P}_{\mathrm{IPC}}(E_3(n)^c) \lesssim n^{-2}.\]

Finally, we have
\[ \mathbf{P}^\omega(W_2(n)^c), \mathbf{P}^\omega(W_3(n)^c) \lesssim n^{-2}\] 
uniformly in $\omega$. Indeed, suppose $x$ and $y$ are two sites as in the description of $W_2(n)$, then, for $n$ sufficiently large,
\[\operatorname{dist}_{\Gamma(n)}(x,y) \le (6q+1)^2\le 40q^2\]
for any $\omega$.
Using 
\[1\le \operatorname{deg}_{\Gamma(n)}(x), \operatorname{deg}_{\Gamma(n)}(y) \le 4\] 
for any $x,y$ in the IPC, we find that on $W_2(n)^c$, for some pair $x,y$ either
\[L(x,\sigma^*(n)) \le \frac{1}{8}L(y,\sigma^*(n)) \le \frac{1}{2}\frac{\operatorname{deg}_{\Gamma(n)}(y)}{\operatorname{deg}_{\Gamma(n)}(x)}L(y,\sigma^*(n))\]
and
\[L(y,\sigma^*(n)) \ge 8 \cdot 320 L_0 q^2\log n \ge 64 L_0 \log n \operatorname{dist}_{\Gamma(n)}(x,y) \]
or
\[L(y,\sigma^*(n)) \le \frac{1}{2}\frac{\operatorname{deg}_{\Gamma(n)}(y)}{\operatorname{deg}_{\Gamma(n)}(x)}L(x,\sigma^*(n))\]
and
\[L(x,\sigma^*(n))) \ge 8L(y,\sigma^*(n)) \ge 8 L_0 \log n \operatorname{dist}_{\Gamma(n)}(x,y).\]
The first case is contained in the event appearing in (\ref{eq: localtimes}). So is the second case, after reversing the roles of $x$ and $y$ in that event. Using \[\operatorname{dist}_{\Gamma(n)}(x,y) \lesssim n^2,\] applying Lemma \ref{continuity}, and taking the union over all pairs, $x,y \in S(3m)\setminus S(m)$, we find that, whatever $\omega$ in the support of $\mathbb{P}_{\mathrm{IPC}}$:
\[\mathbf{P}^\omega(W_2(n)^c)\lesssim \sum_{x,y\in S(3m)\setminus S(m)} \sharp (S(3m)\setminus S(m)) \cdot \exp ( - 8\log n ) \lesssim  n^6 n^{-8}.\]
A similar argument applies to $W_3(n)$.

\subsection{End of the proof}

Applying the Borel-Cantelli lemma to 
\[E_1(n)^c\cup E_2(n)^c \cup E_3(n)^c \cup E_4(n)^c,\]
we find that for $\mathbb{P}_{\mathrm{IPC}}$-almost every $\omega$, there exists $N(\omega)$ such that $\cap_i E_i(n)$ holds when $n\ge N(\omega)$. For any such $\omega$, a further application of the Borel-Cantelli lemma shows that, $\mathbf{P}^\omega$-almost surely, $\cap_i W_i(n)$ holds for $n$ large enough.

It remains to show that whenever all the events above hold, we have the subdiffusive bound of Theorem~\ref{main-theorem}. 
First, on $E_4(n)\cap W_3(n)$, if we denote by $\sum^*$ the sum over indices $\iota$ such that $\mathcal{C}(\iota)$ contains a site $x_\iota \in \tilde{B}(n)$ with
\[L(x_\iota,\sigma^*(n))\ge 320 L_0 q^2\log n,\]
then, assuming $W_1(n)$ also occurs, adjusting the constant $Q$ in the definition of $q$ (see \eqref{eq: qdef}):
\begin{align*}(\sum_\iota -\sideset{}{^*}\sum_\iota) \ \Theta(\iota) &\le 16 \cdot 320L_0q^2 \log n \cdot \sharp \tilde{B}(n) \\
&\le \frac{1}{2}b(n).
\end{align*}
It follows that
\begin{equation}
\label{eq: sumcompare}
\frac{1+\sigma^*(n)}{1+b(n)}\ge \frac{1}{32}\frac{\sum\nolimits^* \Lambda(\iota)}{\sum\nolimits^* \Theta(\iota)}.
\end{equation}
On $W_2(n)$, letting $y_\iota$ be the lexicographically earliest point of $\tilde{B}(n)$ in $\mathcal{C}(\iota)$, we have for those indices $\iota \le \lambda$ occurring in $\sum^*$:
\begin{align*}
\Lambda(\iota)&\ge \frac{1}{8}L(y_\iota,\sigma^*(n))\cdot \sharp\mathcal{C}(\iota)\\
\Theta(\iota) &\le 8L(y_\iota,\sigma^*(n)) \cdot \sharp\tilde{B}(n)\cap \mathcal{C}(\iota).
\end{align*}
For each $\iota\le\lambda$, $F(\mathbf{j}_\iota,q)\setminus D(\mathbf{j}_\iota,q)$ contains an invaded crossing in $\mathcal{C}(\iota)$. Thus, on $E_2(n)\cap E_3(n)$, we can write:
\[\frac{\Lambda(\iota)}{\Theta(\iota)}\ge \frac{1}{8^2 \C_{14}}\frac{\pi(q)}{\rho(q)} \frac{1}{(\log q)^5}, \quad \iota\le\lambda\]
Bounding every $\Lambda(\iota)$ term below individually in (\ref{eq: sumcompare}) and using the BK inequality, we find:
\begin{equation}\label{eq: missing_BK}
\frac{1+\sigma^*(n)}{1+b(n)} \gtrsim \frac{\pi(q)}{\rho(q)}\frac{1}{(\log q)^6} \gtrsim \frac{1}{\pi(q)(\log q)^5}.
\end{equation}
On $W_1(n)$, we have $b(n)\ge n^{2s'}/\log n$. Recalling the definition of $q$ from \eqref{eq: qdef}, we have the following bound for $\pi(q)$. 
\[\pi(q) \lesssim q^{-\eta_1} \lesssim n^{-(1/2)\eta_1\eta_2} (\log n)^{3\eta_1/2}\]
Choosing $\epsilon_0>0$ such that $2 < 2+\epsilon_0 < 2s'$, we obtain:
\[\tau(n)\ge \sigma^*(n) \gtrsim \frac{n^{2s'}}{\log n} \cdot \frac{1}{\pi(q)(\log q)^5} \gg n^{2+\frac{1}{2}\eta_1\eta_2+\epsilon_0}, \]
the desired result.

\section{\large{Derivation of the volume estimates}}
\label{section-volume}
In this section we prove the volume estimates (\ref{eq-backbone-bound}) and (\ref{eq-crossing-bound}). 

\subsection{Estimates for the size of the backbone}
We show that the following moment bounds hold for $\mathbf{j} \in \mathcal{J}$:
\begin{align}
\label{eq-small-moments}
\mathbb{E}_{\mathrm{IPC}}(\sharp F(\mathbf{j},q)\cap \tilde{B}(n))^k & \le k! \cdot (\C_{16} q^2\rho(q))^k\\
\label{eq-big-moments}
\mathbb{E} _{\mathrm{IPC}}(\sharp \tilde{B}(n))^k & \le k! \cdot k^k  \cdot (\C_{17}n^2\rho(n))^k, 
\end{align}
for $k=1,2...$ and constants $\C_{16}$, $\C_{17}$. 

The estimate (\ref{eq-small-moments}) implies the existence, for $\lambda>0$ small enough, of the exponential moment:
\[\mathbb{E}_{\mathrm{IPC}}\exp(\lambda \sharp F(\mathbf{j},q)\cap \tilde{B}(n)) <\infty.\] 
Applying Chebyshev's inequality with $\lambda = 1/(2\C_{16} q^2\rho(q))$ yields ($\ref{eq-backbone-bound}$). 

From (\ref{eq-big-moments}), we obtain the finiteness, for sufficiently small $\lambda$, of the stretched exponential moment:
\[\mathbb{E}_{\mathrm{IPC}}\exp(\lambda \sharp (\tilde{B}(n))^{1/2}) <\infty.\]
Using Chebyshev's inequality with $\lambda = 1/(2(e\C_{17} n^2\rho(n))^{1/2})$, we obtain, for each $c>0$:
\begin{equation}
\label{eq-big-backbone-bound}
\mathbb{P}_{\mathrm{IPC}}\left( \sharp \tilde{B}(n) \ge  \frac{c}{\C_{15}} n^2\rho(n)(\log n)^2\right)\lesssim n^{-c^{1/2}}.
\end{equation}

To derive (\ref{eq-small-moments}) and (\ref{eq-big-moments}), we follow the method introduced by J\'arai \cite{jarai} to estimate the moments of the volume $|\Lambda|$ of the IPC in a box. We will instead apply this argument to the volume of a backbone, and then combine it with an inductive argument of Nguyen \cite{nguyen}.

We begin with the first moment ($k=1$) in (\ref{eq-small-moments}). If $F(\mathbf{j},q)\subset S(n)\setminus (S(m))^\circ$, the number of sites of $F(\mathbf{j},q)$ with two disjoint connections in the IPC to $\partial F(\mathbf{j},q)$ provides an upper bound for the volume of $F(\mathbf{j},q)\cap \tilde{B}(n)$.  Let $Z_q(\mathbf{j},j)$ denote the set of sites in $F(\mathbf{j},q)$ with two $p_{2m}(j)$-open connections to $\partial F(\mathbf{j},q)$.  Note that 
\[\tilde{B}(n) \subset S(m)^c.\]
On $H_{2m}(j)$ (defined in \eqref{eq:jarai-bound}), every edge of the IPC in $S(m)^c$ is $p_{2m}(j)$-open, as noted at the end of Section \ref{section-invasion}, and thus:
\begin{align} \label{eq: small-moment-sum} \mathbb{E}(\sharp(\tilde{B}(n) \cap F(\mathbf{j},q)))  &\le \mathbb{E}(\sharp (\tilde{B}(n)\cap F(\mathbf{j},q)) ;H_{2m}(1)^c) \nonumber\\
&+\sum_{j=2}^{\log^*2m}\mathbb{E}(\sharp Z_q(\mathbf{j},j-1);H_{2m}(j-1)\cap H_{2m}(j)^c)\\
&+\mathbb{E} (\sharp Z_q(\mathbf{j},\log^*2m)).\nonumber
\end{align}
The first term is bounded up to a constant factor by: 
\[(3q)^2\cdot \mathbb{P}(H_{2m}(1)^c) \lesssim q^2(2m)^{-\C_0M}.\]
The terms of the sum are estimated using the Harris-FKG inequality:
\begin{align}
\mathbb{E}(\sharp Z_q(\mathbf{j},j-1);H_{2m}(j-1)\cap H_{2m}(j)^c) &\le \mathbb{E}(\sharp Z_q(\mathbf{j},j-1)) \cdot \mathbb{P}(H_{2m}(j)^c) \nonumber \\
&\lesssim \mathbb{E}(\sharp Z_q(\mathbf{j},j-1)) \exp(-M\C_0 \log^{(j)} 2m) \label{eq: mac_and_cheese}
\end{align}
By decomposing $F(\mathbf{j},q)$ according to the distance $l$ to $\partial F(\mathbf{j},q)$, we find:
\begin{align}
\label{eq: nsquared}
\mathbb{E}\sharp (Z_q(\mathbf{j},j-1)) & \lesssim \sum_{l=1}^{3q/2} q \rho(p_{2m}(j-1),l) \\
&\lesssim \sum_{l \le \lfloor L(p_{2m}(j-1))\rfloor } q \rho(p_{2m}(j-1),l) \label{eq: smoothsum} \\
&+\  q^2\rho(p_{2m}(j-1),L(p_{2m}(j-1))) \cdot \mathbf{1}[\ 3q/2 > \lfloor L(p_{2m}(j)) \rfloor\ ]. \nonumber
\end{align}
By the same argument used for (7) in \cite{kesten2} (see Remark (37) there), the sum up to $L(p_{2m}(j-1))$ in (\ref{eq: smoothsum}) is bounded up to a constant by 
\[q^2\rho(p_{2m}(j-1),L(p_{2m}(j-1))).\] 
The proof in \cite{kesten2} is carried out for $p = p_c$, but the implicit constants that appear are due to applications of RSW theory and thus are uniformly bounded in $p>p_c$. By comparability of the arm exponents below $L(p)$ \cite{kesten3} (see also \cite[Theorem~26]{nolin}), we have
\[\rho(p_{2m}(j-1),L(p_{2m}(j-1)) \lesssim \rho(p_c, L(p_{2m}(j-1))).\]
Thus, finally, in (\ref{eq: nsquared}), we have (since $q\le 2m$)
\begin{align}
\label{eq: rholessthanlog}
\mathbb{E}\sharp (Z_q(\mathbf{j},j-1)) &\lesssim  q^2\rho(p_c,2m) M \log^{(j-1)}2m\\
&\lesssim q^2\rho(p_c,q) M \log^{(j-1)}2m, \nonumber
\end{align}
where in the first step we have used the inequality
\begin{equation}
\label{eq:twoarms}
\frac{\rho(p_c,r)}{\rho(p_c,s)}\gtrsim \frac{s}{r}
\end{equation}
for $r\le s$. A similar inequality for $\pi(p_c,n)$ was used in \cite{jarai}, where the author indicates that it can be proved by the argument in \cite[Corollary~3.15]{vdbkesten}. The proof of (\ref{eq:twoarms}) follows the same general strategy, but does not use the van den Berg-Kesten inequality: let $k=0,1,2, \ldots, \lceil r/s \rceil$, $v_k=\mathbf{0}\pm 2ks$ and consider the annuli $S(v_k,r)\setminus S^\circ(v_k,s)$. The inner squares of these $2\lceil r/s \rceil+1$  annuli are adjacent. The event that there exists a $p_c$-open left-right crossing of $S(r)$ has probability bounded below uniformly in $r$, and implies that one of the annuli is crossed by two disjoint $p_c$-open paths. By quasi-multiplicativity, this probability is comparable to $\rho(r)/\rho(s)$, and (\ref{eq:twoarms}) follows by a union bound.

Inserting (\ref{eq: rholessthanlog}) into \eqref{eq: mac_and_cheese} and then into (\ref{eq: small-moment-sum}), we find
\begin{multline*}
\mathbb{E}(\sharp (\tilde{B}(n)\cap F(q,\mathbf{j})) )\lesssim q^2\rho(p_c,q) \\ \times \left(\frac{\exp(-M\C_0\log 2m)}{\rho(p_c,q)}+ M \sum_{j=2}^{\log^* 2m}(\log^{(j-1)} 2m)^{-M\C_0} \log^{(j-1)} 2m +M\right)
\end{multline*}

The final term corresponds to $\mathbb{E}(\sharp Z_q(\mathbf{j},\log^*2m))$, which is $O(q^2 \rho(p_c, q))$ by \eqref{eq: rholessthanlog}. Using \eqref{eq:twoarms}, we may choose $M$ large enough to make the first term $O(1).$ An important point made in \cite{jarai} is that choosing $M$ possibly larger, we may bound the contribution from the sum in the parentheses by a constant. Indeed, we have
\begin{equation}
\label{eq:logsconverge}\sup_{n\ge 1} \sum_{j=1}^{\log^*n}\left(\log^{(j)} n\right)^{-1} \lesssim 1.
\end{equation}

This establishes (\ref{eq-small-moments}) for $k=1$. To deal with the higher moments, we use the following general lemma:

\begin{lemma}\label{nguyenslemma}
Let $p_c\le p\le 1$, $n\ge1$, and $C_n(p)$ be the set of sites of $S(n)$ with two disjoint $p$-open connections to $\partial S(n)$.  There exists a constant $\C_{18}$ independent of $n$ and $p$ such that, for any $k\ge 1$, the following inductive bound holds:
\[\mathbb{E}(\sharp C_n(p))^{k+1} \le \C_{18} (k+1) n^2 \rho(p,n) \mathbb{E}(\sharp C_n(p))^k.\]
\begin{proof}
The result is essentially due to Nguyen \cite{nguyen}, who proved that for $p\ge p_c$ and $L\le L(p)$,
\[\mathbb{E}(\sharp W_L)^{k+1} \le  \C_{19} (k+1) L^2\pi(p,L)\mathbb{E}(\sharp W_L)^k, \quad k\ge 1,\]
where $W_L$ is the set of sites in $S(L)$ connected to $\partial S(L)$ by a $p$-open path. $\C_{19}$ is a constant uniform in $k, L$ and $p$. 

When $n \le L(p)$, the proof in  \cite{nguyen} is easily adapted to the variables $\sharp C_n(p)$. We define the event
\[A(x) = \{x \text{ has two disjoint open connections to } \partial S(n) \}.\]
The idea is to write
\begin{align}
\nonumber \mathbb{E}(\sharp C_n(p))^{k+1} &=\sum_{x_1,\ldots, x_{k+1}\in S(n)} \mathbb{P}_{p}(\cap_{i=1}^{k} A(x_i), A(x_{k+1}))\\
\label{eq: nguyen-triple-sum} &= \sum_{l=1}^{n/2} \sum_{x_1,\ldots,x_k\in S(n)}\sum_{x_{k+1}\in R_l \cap S(n)} \mathbb{P}_{p}\left(\cap_{i=1}^{k}A(x_i), A(x_{k+1})\right),
\end{align}
where we have set
\[R_l= R_l(x_1,\ldots x_k) = \{x: \operatorname{dist}_{\infty}(x,\{x_1,\ldots,x_k\}\cup \partial S(n)) = l\}.\]
Letting 
\[\mathrm{Circ}_{k,l} = \left\{ \begin{gathered}\text{there exists an open circuit}\\ \text{ around } x_{k+1} \text{ in } S(l,x_{k+1})\setminus S(l/2,x_{k+1}) \end{gathered} \right\}, \]
we have $\mathbb{P}(\mathrm{Circ}_{k,l})\ge \C_{20} >0$ uniformly in $l$ and $n$ (even for $l> L(p)$). By the FKG inequality:
\begin{align*}
\mathbb{P}_{p}\left(\cap_{i=1}^{k}A(x_i), \ A(x_{k+1})\right) \le& \frac{1}{\C_{20}}\mathbb{P}_{p}\left(\cap_{i=1}^{k}A(x_i), A(x_{k+1}), \mathrm{Circ}_{k,l}\right)\\
\le& \frac{1}{\C_{20}}\mathbb{P}_{p}\left(\begin{gathered} \cap_{i=1}^{k}\tilde{A}(x_i,l), x_{k+1} \text{ has two disjoint} \\ \text{connections to } \partial S(l/2,x_{k+1}) \end{gathered}  \right).
\end{align*}
$\tilde{A}(x_i,l)$ is the event that $x_i$ is connected to $\partial S(n)$ by two disjoint open paths outside of $S(l/2,x_{k+1})$. By independence, the last quantity on the right is bounded, up to a constant, by
\[\mathbb{P}_{p}\left(\cap_{i=1}^{k}A(x_i)\right)\rho(p,l/2).\]
For any $l$, we have:
\[\sharp R_l \lesssim (k+1)\cdot n.\]
Returning to (\ref{eq: nguyen-triple-sum}), we find
\begin{equation}\label{eq: nguyensum}
\mathbb{E}(\sharp C_n(p))^{k+1} \lesssim (k+1) \cdot \mathbb{E}(\sharp C_n(p))^k\cdot n\sum_{l=1}^{n/2} \rho(p,l/2).\end{equation}
If $n \le L(p)$, we have
\[\rho(p,l/2) \asymp \rho(p_c,l/2) \asymp \rho(p_c,l),\]
and the estimate (see the remark concerning (\ref{eq: smoothsum}) above)
\[\sum_{l=1}^{n/2} \rho(p_c,l) \lesssim n\rho(p_c,n) \lesssim n \rho(p,n) \]
leads to the inductive estimate claimed above.

If $n\ge L(p)$, we split the sum as we did in the treatment of the first moment of $\sharp Z_q$:
\begin{align}
\sum_{l=1}^{n/2} \rho(p,l/2) &= \left(\sum_{l=1}^{\lfloor L(p)\rfloor }+\sum_{l=\lfloor L(p)\rfloor+1}^{n/2}\right)\rho(p,l/2)\\
&\lesssim \sum_{l=1}^{\lfloor L(p)\rfloor }\rho(p,l/2) + n\rho(p, L(p) /2) \nonumber \\
&\lesssim n\rho(p,L(p)) + n \rho(p,L(p)/2)
\nonumber \\
&\lesssim n \rho(p, n) \nonumber.
\end{align}
Here we have used that for $L \geq L(p)$, $\rho(p,L) \asymp \rho(p,L(p))$. This follows by a variant of the argument presented in \cite[Section 7.4]{nolin} . This establishes the lemma.
\end{proof}
\end{lemma}
Using Lemma~\ref{nguyenslemma}, induction and the fact that $q \leq 2m/(DM\log 2m)$ for large $m$, we obtain
\begin{equation}\label{eq: inductiveboundforn}  \mathbb{E} (\sharp Z_q(\mathbf{j},j))^k \le k!(\C_{18} q^2\rho(p_{2m}(j),q))^k \le k! (\C_{18} q^2 \rho(p_c,q))^k.
\end{equation}

Thus, arguing as for (\ref{eq: small-moment-sum}):
\begin{multline*}
\mathbb{E}((\sharp (\tilde{B}(n)\cap F(q,\mathbf{j}))^k )\lesssim k! ~q^{2k}(\rho(p_c,q))^k (\C_{18})^k\\ \times \left(\frac{\exp(-M\C_0\log 2m)}{k!(\C_{18}\rho(p_c,q))^k}+ \sum_{j=2}^{\log^* 2m}(\log^{(j-1)} 2m)^{-M\C_0} +1\right).
\end{multline*}
Using the value of $q$ from \eqref{eq: qdef} and choosing $M= \eta_2 k/2\C_0$, we use (\ref{eq:logsconverge}) to get (\ref{eq-small-moments}) in the case where $\mathbf{j}$ is such that $F(q,\mathbf{j})\subset S(3m)\setminus S(m)^\circ$. For a general $\mathbf{j} \in \mathcal{J}$, the intersection
\[F(q,\mathbf{j}) \cap (S(3m)\setminus S(m)^\circ)\]
is a union of at most two rectangles with side lengths $r_1$ and $r_2$, $r_i\le 3q/2$. Repeating the arguments above, we see that the size of the intersection of each of these rectangles with the backbone $\tilde{B}(n)$ enjoys the moment bounds (\ref{eq-small-moments}), with $q^2\rho(p_c,q)$ replaced by $r_1r_2\rho(p_c,r)$, with $r=\max\{r_1,r_2\}$. Using (\ref{eq:twoarms}), we obtain the upper bound $r^2\rho(p_c,r)$. For the higher moments, moment bounds of the form (\ref{eq-small-moments}) with a larger (but still uniform in $\max(r_1,r_2)\le 3q/2$) constant $\C_{16}$ are valid. (\ref{eq-backbone-bound}) now follows by a union bound.

The proof of (\ref{eq-big-moments}) follows a very similar pattern to the above. Instead of $Z_q(\mathbf{j},j)$, we consider the sets 
\[Z_n(j), \quad 1\le j \le \log^*2m\] 
of points of $S(3m)\setminus S(m)$ with two disjoint $p_{2m}(j)$-open connections to $\partial S(m) \cup \partial S(3m)$. Repeating the steps for the case $k=1$ gives \eqref{eq-big-moments} for the first moment. In the case of higher moments, we need to modify inequality \eqref{eq: inductiveboundforn}, replacing it with
\[
\mathbb{E}(\sharp Z_n(j))^k \leq k!(\C_{18}n^2\rho(p_{2m}(j),n))^k \leq k!(M \C_{21} n^2 \rho(p_c,n) \log^{(j)}2m)^k
\]
for some $\C_{21}$. Decomposing as before over the events $(H_{2m}(j))$ and choosing $M\geq (k+1)/\C_0$ leads to \eqref{eq-big-moments}.

\subsection{Proof of (\ref{eq-crossing-bound})}
In its general outline, the proof is similar to that of (3.24) in \cite{kesten}, with some parameters chosen differently because we wish to bound only logarithmic deviations from the mean. However, the estimates in \cite{kesten} are carried out for critical percolation ($p=p_c$), and the proof of the initial estimate (\ref{eq:pre-peierls-bound}) below in the supercritical case introduces an additional technical difficulty. 

As explained in Section \ref{section-invasion}, the entire (finite) $p_c$-open cluster of any site in the IPC also belongs to the IPC. Thus, for any crossing $r$ (in the IPC) of $F(\mathbf{j},q)\setminus D(\mathbf{j},q)$, the number of sites in $F(\mathbf{j},q)\setminus D(\mathbf{j},q)$ connected to $r$ by $p_c$-open paths provides a lower bound for the quantity in (\ref{eq-crossing-bound}).

The starting point is the following:
\begin{lemma}[\cite{kesten2}]
Let $r$ be a deterministic path crossing $S(27n)\setminus (S(n))^\circ$.
Let  $Z(n)$ be the set of sites $p_c$-connected to $r$ inside this annulus. We have the lower bound:
\[\mathbb{P}_{p_c}(\sharp Z(n) \ge \C_{22} n^2\pi(n))\ge \C_{23},\]
for some constants $\C_{22},\C_{23}>0$ independent of $n$.
\end{lemma}

The proof is essentially that of (56) in \cite{kesten2}. Kesten's idea is to compute the first and second moments of the number $Y(n)$ of sites in $S(9n)\setminus (S(3n))^\circ$ connected to open circuits in $S(3n)\setminus (S(n))^\circ$ and $S(27n) \setminus (S(9n))^\circ$ (and thus to $r$) and use the Harris-FKG inequality and the second moment method.

Fix some $\delta > 0$ (to be chosen later). We first show that, for any $0<t<q$ (entailing in particular $t < L(p_n(1))$ by \eqref{eq: twoarmbound}), and any coordinates $\mathbf{v}=(v_1,v_2)$ such that 
\[T(\mathbf{v})=[-t+v_1,t+v_1]\times [-t+v_2,4t+v_2]\subset S(3m)\setminus (S(m))^\circ,\]
we have, for some constants $\C_{24}$, $\C_{25}>0$,

\begin{multline}
\label{eq:pre-peierls-bound}
\mathbb{P}\left(\begin{gathered}\exists \text{ a } p_n(1)\text{-open crossing } r  \text{ of } J(\mathbf{v}) = [v_1,t+v_1]\times [v_2,3t+v_2] \\\text{such that }  \sharp \{x \in T(\mathbf{v}) :x \xrightarrow{p_c} r \text{ in } T(\mathbf{v})\} \le \C_{24} t^2\pi(t)/(\log t)^\delta \end{gathered}\right)\\ \lesssim \frac{1}{(\log t)^{\C_{25}}}.
\end{multline}

\subsection{Proof of (\ref{eq:pre-peierls-bound})}
For any crossing $r$ of $J(\mathbf{v})$, let
\[Z(T(\mathbf{v}),r) = \{x \in T(\mathbf{v}) :x \text{ is connected to } r \text{ in } T(\mathbf{v}) \text{ by a } p_c\text{-open path}\}.\]
The probability on the left of (\ref{eq:pre-peierls-bound}) equals
\begin{multline}
\label{eq:crossings-union-bound}
\mathbb{P}\left(\exists \text{ a } p_n(1)\text{-open crossing } r: \ \sharp Z(T(\mathbf{v}),r)\le \frac{\C_{24} t^2 \pi(t)}{(\log t)^\delta} \right) \\
\le \mathbb{P}\left(\exists \text{ a } p_c\text{-open crossing } r': \ \sharp Z(T(\mathbf{v}),r')\le \frac{\C_{24} t^2 \pi(t)}{(\log t)^\delta}\right)\\
+\mathbb{P}\left(\begin{gathered} \exists \text{ a }p_n(1)\text{-open crossing } r \text{ such that } r\\ \text{ intersects no } p_c\text{-open crossing of } J(\mathbf{v})\end{gathered}\right).
\end{multline}
The precise meaning of ``$r$ intersects no $p_c$-open crossing of $J(\mathbf{v})$'' is that no site in $J(\mathbf{v})$ is a common endpoint of an edge in $r$ and an edge in some horizontal $p_c$-open crossing of $J(\mathbf{v})$. In particular, $r$ is edge-disjoint from all $p_c$-open crossings.

Both terms on the right in (\ref{eq:crossings-union-bound}) will be bounded, up to a constant factor, by $(\log t)^{-\delta \C_{25}}$. We begin by estimating the first term in (\ref{eq:crossings-union-bound}). For any crossing lattice path $r'$ of $J(\mathbf{v})$, let $J^-(r')$ be the set of edges with an endpoint that can be connected to $[v_1,v_1+t] \times \{v_2\}$ by a path in $J(\mathbf{v})$ that does not touch $r'$ (below $r'$). Note that $J^-(r')$ may include edges not entirely contained in $J(\mathbf{v})$. The lowest $p_c$-open crossing  $R_1$ of $J=J(\mathbf{v})$ is defined as the horizontal crossing of the rectangle by $p_c$-open edges such that the component $J^-(R_1)$, is minimal. $R_k$ is defined inductively as the lowest crossing of $J\setminus (J^{-}(R_{k-1})\cup R_{k-1})$ (defined analogously -- see \cite[Prop.~2.3]{kestenbook} for the existence of $R_k$ and precise definitions). For a given (lattice path) crossing $r'$ of $J(\mathbf{v})$, write $\Sigma_{r'}$ for the sigma algebra generated by the status of edges in $r' \cup J^-(r')$. We define $K$ to be the maximal $k$ such that $R_k$ exists. The veracity of the following string of inequalities is then evident:
\begin{align}
&\mathbb{P}\left(\exists \text{ a } \ p_c\text{-open crossing } r': \ \sharp Z(T(\mathbf{v}),r')\le \frac{\C_{24} t^2 \pi(t)}{(\log t)^\delta}\right) \nonumber \\
\le&\, \sum_{k\ge 1} \mathbb{P}\left(\ \sharp Z(T(\mathbf{v}),R_k)\le \frac{\C_{24} t^2 \pi(t)}{(\log t)^\delta} ; K \ge k \right) \nonumber \\
\le& \sum_{k\ge 1} \sum_{r''} \mathbb{E} \left(\mathbb{P}(\ \sharp Z(T(\mathbf{v}),R_k)\le \C_{24} t^2 \pi(t)/(\log t)^\delta \mid \Sigma_{r''}); R_k = r'', K\ge k)\right) \label{eq: doublesum}.
\end{align}
On $\{R_k = r'', K\ge k\}$, we have the following uniform estimate for the conditional probability given $\Sigma_{r''}$:
\begin{equation}
\label{eq: oneoverlogdelta}
\mathbb{P}(\sharp Z(T(\mathbf{v}),r'')\le \C_{24} t^2 \pi(t)/(\log t)^\delta \mid \Sigma_{r''}) \lesssim \frac{1}{(\log t)^{\C_{25}}}.
\end{equation}
To see this, consider the left endpoint $l_{r''}$ of the crossing $r''$, and annuli 
\[A(l_{r''},3^k) = S(l_{r''},3\cdot 3^k)\setminus S(l_{r''},3^k), \quad \frac{t}{(\log t)^{\delta/2}}\ \le 3^k\le t.\]
For $3^{3j} \leq t/27$, the existence of circuits $C'_j$ around $l_{r''}$ in $A(l_{r''},3^{3j})$ and $C''_j$ in $A(l_{r''},3^{3j+2})$, all of whose edges outside $J^-(r'')$ are $p_c$-open implies that any site in $A(l_{r''},3^{3j+1}) \cap ([-t,0)\times \mathbb{R})$ connected to
\[\partial S(l_{r''},3^{3j}) \cup \partial S(l_{r''},3^{3j+3})\]
is $p_c$-connected to the crossing $r''$. Thus, using the Harris-FKG inequality, independence of the edge configurations in $J^-(r'')$ and $[-t,0)\times \mathbb{R}$ and the second moment method as in the discussion preceding (\ref{eq:pre-peierls-bound}), there exist constants $\C_{24}$, $\C_{26}$, such that for each $j$ with $ t/(\log t)^{\delta/2} \le 3^{3j} \le t/27$:
\begin{align*}
 &\mathbb{P}\left(\sharp Z(T(\mathbf{v}),r'')\le \C_{24} t^2\pi(t) /(\log t)^\delta \mid \Sigma_{r''} \right) \\
 \le&\  \mathbb{P}\left(\sharp \{x\in A(l_{r''},3^{3j+1})\cap ([-t,0)\times \mathbb{R}) : x \xrightarrow{p_c} r'' \}\le \C_{24} 3^{6j}\pi(3^{3j})\mid \Sigma_{r''} \right)\\
 \le &\  1-\C_{26}.
\end{align*}
There are $(\delta/(6\log 3))\log\log t +O(1)$ admissible indices $j$, and so by independence of the configuration in the different annuli, we find
\[
\mathbb{P}\left(\sharp Z(T(\mathbf{v}),r'')\le \C_{24} t^2/(\log t)^\delta \mid \Sigma_{r''}\right)
\lesssim (1-\C_{26})^{(\delta/(6\log 3))\log\log t},
\]
which is the same as (\ref{eq: oneoverlogdelta}).

Returning to the double sum of \eqref{eq: doublesum}:
\begin{align}
&\mathbb{P}\left(\exists\  \text{a } p_c\text{-open crossing } r : \ \sharp Z(T(\mathbf{v}),r)\le \frac{\C_{22} t^2 \pi(t)}{(\log t)^\delta}\right) \nonumber \\ 
& \lesssim \sum_{k\ge 1} \sum_r \frac{1}{(\log t)^{\C_{25}}} \mathbb{P} \left( R_k = r, K\ge k\right) \nonumber \\
&= \frac{1}{(\log t)^{\C_{25}}} \sum_{k\ge 1}  \mathbb{P}\left( K\ge k\right) \label{eq: k-sum}
\end{align}
By the Russo-Seymour-Welsh method, the $\mathbb{P}_{p_c}$ probability of a dual vertical crossing of $J(\mathbf{v})$ is bounded below by some $\epsilon >0$. Thus, by disjointness of the $R_k$'s and the BK inequality,
\[
\mathbb{P}\left( K\ge k\right) \le \mathbb{P}(\exists\, k \text{ disjoint $p_c$-open crossings of } J(\mathbf{v})) \le (1-\epsilon)^k.
\]
This allows us to bound the sum in (\ref{eq: k-sum}) by $C/\epsilon$.

We now estimate the second term on the right in (\ref{eq:crossings-union-bound}). Denote by $\Xi$ the event that there exists a $p_n(1)$-open crossing $r$ of $J(\mathbf{v})$ such that $r$ intersects no $p_c$-open crossing of $J(\mathbf{v})$. For any $K_0>0$, we have
\[ \mathbb{P}(\Xi) \le \mathbb{P}(\Xi, K \le K_0)+\mathbb{P}(K> K_0).\]
As previously, $K$ denotes the maximal number of disjoint $p_c$-open crossings of $J(\mathbf{v})$. We will choose $K_0=c\log \log n$, so as to give the following bound for the second term above:
\[\mathbb{P}(K> K_0)\lesssim \exp(-\C_{27} \log\log n) = (\log n)^{-\C_{27}},\]
where the constant $\C_{27}$ is a constant such that $\C_{27}\ge \C_{25}$.

For the first term, we have the union bound
\[\mathbb{P}(\Xi, K \le K_0) \le \sum_{k=0}^{\lceil c \log\log n\rceil} \mathbb{P}(\Xi, K=k).\]
It will be shown below (see Lemma \ref{crossing-lemma}) that there is a constant $\C_{28}$ such that, for any $\mathbf{v}$ with $T(\mathbf{v})\subset S(3m)\setminus (S(m))^\circ$, 
\begin{equation} 
\label{eq: fourarms}
\mathbb{P}(\Xi, K=k) \le (\C_{28}\log t)^{2k} (p_n(1)-p_c)\cdot t^2 \cdot \pi_4(t),
\end{equation}
where $ \pi_4(t)= \pi_4(t,p_c)$ is the ``alternating 4-arm probability,'' associated to the event that $\langle \mathbf{0},\mathbf{e}_1\rangle$ is connected to $\partial S(t)$ by two disjoint $p_c$-open paths and its dual edge is connected to $\partial S(t)$ by two disjoint $p_c$-closed dual paths. Thus
\begin{align}
\mathbb{P}(\Xi, K\le K_0) &\lesssim (\log \log n) \exp(2K_0\log(\C_{28}\log t)) (p_n(1)-p_c)\cdot t^2 \cdot \pi_4(t) \nonumber \\
&\le \exp\left(\C_{29} (\log \log n)^2\right) \cdot (p_n(1)-p_c)\cdot t^2 \cdot \pi_4(t), \label{eq: bound-for-pgek}
\end{align}
for a constant $\C_{29}$. The factor $(p_n(1)-p_c)\cdot t^2 \cdot \pi_4(t)$ is $O(n^{-c}$). Indeed, it was shown in \cite{kesten3} that, uniformly for $p>p_c$ sufficiently close to $p_c$:
\begin{equation}
\label{eq: nearcritical}
L(p)^2  \pi_4(L(p),p_c)(p-p_c) \asymp 1.
\end{equation}
Applying this to $p=p_t(1)$, and using (\ref{eq:l_p}) and $\pi_4(t/(M\log t)) \asymp \pi_4(t/\log t)$ \cite[Proposition~12]{nolin}, we find, for $t$ large enough:
\[\frac{t^2}{(\log t)^2} \pi_4(t/\log t) \cdot (p_t(1)-p_c)\asymp 1.\]
Thus we have
\[
(p_n(1)-p_c)\cdot t^2 \cdot \pi_4(t) \lesssim \frac{p_n(1)-p_c}{p_t(1)-p_c} \cdot  (\log t)^2.
\]
Here we have used $\pi_4(t) \leq \pi_4(t/\log t)$. Using (\ref{eq: nearcritical}) again, we have:
\[\frac{p_n(1)-p_c}{p_t(1)-p_c} \asymp \frac{L(p_t(1))^2}{L(p_n(1))^2} \cdot \frac{ \pi_4(L(p_t(1))}{ \pi_4(L(p_n(1)))}.\]
By quasimultiplicativity \cite[Proposition~12]{nolin}:
\[ \frac{ \pi_4(L(p_t(1)))}{ \pi_4(L(p_n(1)))} \asymp \frac{1}{ \pi_4\left(L(p_t(1)),L(p_n(1))  \right) },\]
where $\pi_4(n,N)=\pi_4(n,N;p_c)$ is the probability that there are four arms of alternating occupation status connecting $\partial S(n)$ to $\partial S(N)$ in $S(N)\setminus S(n)^\circ$.
Using Reimer's inequality \cite{reimer} and the (exact) scaling of the 5-arm exponent (see \cite[Theorem~23]{nolin} or \cite{werner}), we have:
\begin{align*}
\left(\frac{L(p_t(1))}{L(p_n(1))}\right)^2 &\asymp  \pi_5\left(L(p_t(1)),L(p_n(1))  \right)\\
 &\lesssim \pi\left(L(p_t(1)),L(p_n(1))\right)\cdot \pi_4\left(L(p_t(1)),L(p_n(1))\right).
\end{align*}
Here, $\pi(n,N)$ is the one-arm probability, that $\partial S(n)$ is connected to $\partial S(N)$ by an open path. Since the one-arm probability satisfies the power-type upper bound $\pi(n)\lesssim n^{-\eta_1}$ for some $\eta_1 \leq 1/2$ (apply the BK inequality to the bound on $\eta_2$ in \eqref{eq: twoarmbound}), we find that $(p_n(1)-p_c)\cdot t^2 \cdot \pi_4(t)$ is bounded, up to a constant, by
\begin{equation}
\label{eq: powerbound}
(\log t)^2 \left(\frac{L(p_t(1))}{L(p_n(1))}\right)^{\eta_1} \lesssim (\log n)^2 \left(\frac{t}{n}\right)^{\eta_1}.
\end{equation}
Since we assume $t< q$, and $q=o(n^{1/2})$ (see (\ref{eq: twoarmbound})), we find 
\[(p_n(1)-p_c)\cdot t^2 \cdot (\log t)^{2} \cdot \pi_4(t) = O(n^{-c}),\]
for some $c>0$. Returning to (\ref{eq: bound-for-pgek}), we have the bound:
\[\mathbb{P}(\Xi, K\le K_0) \lesssim n^{-c/2}.\]

It remains to prove (\ref{eq: fourarms}). This is done in Lemma \ref{summation-lemma} below. Before proceeding, let us introduce a definition: A \emph{$p_c$-closed arm with $k$ defects} is a path of dual edges, all of which except for $k$ of them are $p_c$-closed. The proof of Lemma \ref{summation-lemma} depends on the following:
\begin{lemma}\label{crossing-lemma}
Let $\Xi$ be the event that there exists a $p_n(1)$-open crossing $r$ of $J(\mathbf{v})$ such that $r$ intersects no $p_c$-open crossing of $J(\mathbf{v})$, and $K$ be the maximal number of horizontal $p_c$-open crossings of $J(\mathbf{v})$. Suppose $K=k$; then there exists an edge $e\in J(\mathbf{v})$ such that \begin{enumerate}
\item $e$ has two disjoint $p_n(1)$-open arms to $\{v_1\}\times [v_2,v_2+3t]$ (the left side of $J(\mathbf{v})$) and $\{v_1+t\}\times [v_2,v_2+3t]$ (the right side of $J(\mathbf{v})$), respectively.  
\item $e^*$, the dual edge to $e$, has two disjoint $p_c$-closed arms, each with at most $k$ defects to $[v_1,v_1+t]\times \{v_2\}$ (the bottom side of $J(\mathbf{v})$) and to $[v_1,v_1+t]\times \{v_2+3t\}$ (the top side of $J(\mathbf{v})$), respectively.
\item $w(e) \in [p_c,p_n(1)]$.
\end{enumerate}
\begin{proof}
On the event $\{K=k\}$, Menger's theorem \cite[Section~3.3]{diestel} implies that there is a dual path $s$ joining the top of $J(\mathbf{v})$ to the bottom, all of whose edges, with exactly $k$ exceptions, are closed and which moreover does not intersect itself. This path must intersect the horizontal $p_n(1)$-open crossing $r$ \cite[Prop. 2.2]{kestenbook} along a $p_n(1)$-open edge $e$. This edge must then be $p_c$-closed. The dual edge $e^*$, being part of the non-self-intersecting $s$ with $k$ defects, has two dual arms joining it to the top and bottom of $J(\mathbf{v})$. (See Figure~\ref{fig: menger}.) Moreover, the total number of defects on these two arms is $k$. This establishes the lemma.
\end{proof}
\end{lemma}

\begin{figure}[h]
\centering
\includegraphics[scale=0.4]{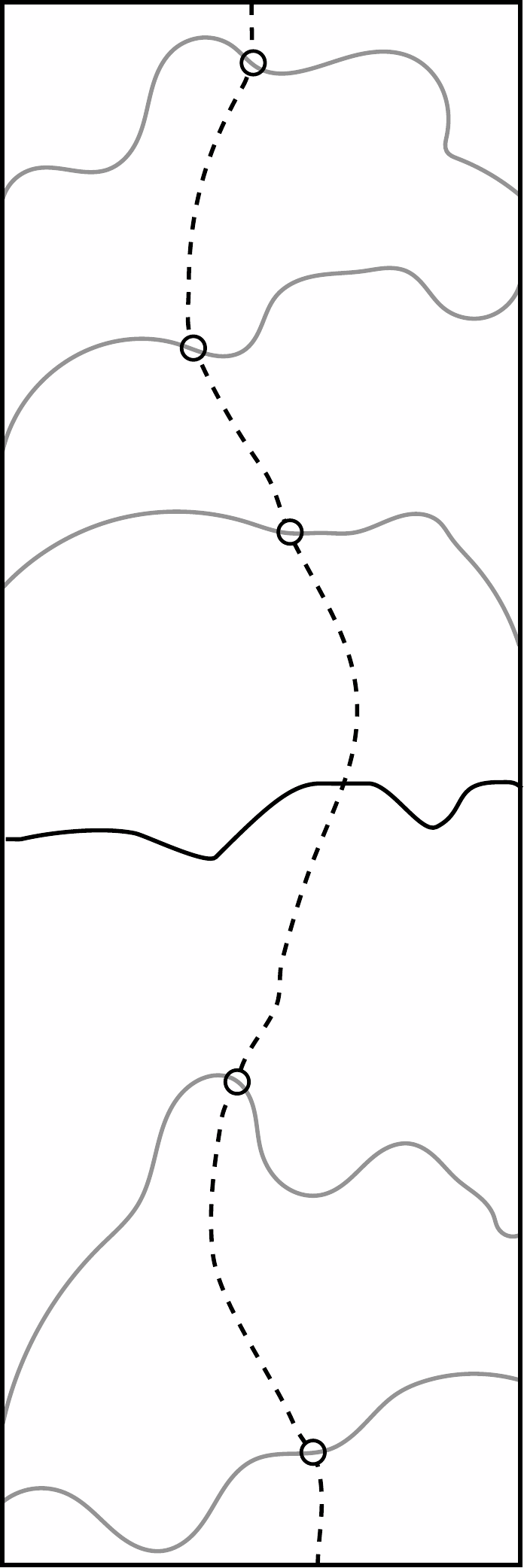}
\caption{Depiction of the application of Menger's theorem in the proof of Lemma~\ref{crossing-lemma}. The dotted path has $k$ defects, shown as empty circles. The solid black path represents a $p_n(1)$-open crossing and the grey paths represent disjoint $p_c$-open crossings.}
\label{fig: menger}
\end{figure}

The proof of (\ref{eq:pre-peierls-bound}) is concluded by the following lemma, which establishes the estimate (\ref{eq: fourarms}):
\begin{lemma}\label{summation-lemma} There is a constant $\C_{28}$ such that, for each $k\ge 1$, the following bound holds:
\begin{equation}
\label{eq: xibound}
\mathbb{P}(\Xi, K=k) \le (\C_{28}\log t)^{2k} (p_n(1)-p_c)\cdot t^2 \cdot \pi_4(t).
\end{equation}
\end{lemma}
It suffices to estimate the probability that there is an edge in $J(\mathbf{v})$ satisfying the two conditions in Lemma \ref{crossing-lemma}. To that end, we will show that the expected number of such edges in $J(\mathbf{v})$ is bounded by the quantity on the right side of equation (\ref{eq: xibound}). For $e\in J(\mathbf{v})$, let $A^k_e$ be the event that $e$ satisfies the conditions of Lemma \ref{crossing-lemma}. 

The key step is the existence of a constant $\C_{29}$ such that
\begin{equation}\label{eq: ecomp} \mathbb{P}\left(A^k_e \right) \le (\C_{29}\log t)^{2k} (p_n(1)-p_c) \mathbb{P}(B_e),
\end{equation}
where $B_e$ is the event that $e$ has two disjoint $p_c$-open arms joining it to the left and right sides of $J(\mathbf{v})$ respectively, and $e^*$ has two disjoint $p_c$-closed dual arms to the top and bottom sides of $J(\mathbf{v})$. The effect of the arms with defects is to produce the logarithmic factor indicated in the equations above:
\begin{equation}\label{eq: defectbound}
\mathbb{P}\left(A^k_e\right) \le (\C_{29}\log t)^{2k}(p_n(1)-p_c)\mathbb{P}(A_e),
\end{equation}
where $A_e$ is defined analogously to $B_e$ above, except that the open connections are required to be $p_n(1)$-open rather than $p$-open.
 This follows from the argument in \cite[Prop. 17]{nolin}, where it is shown that if $A^d_{j,\sigma}(n)$ denotes the probability that the origin is connected to $\partial S(n)$ by $j$ paths, with $d$ defects in total, whose occupation status is specified by the sequence $\sigma \in \{\text{open},\text{closed}\}^j$, then
 \begin{equation}
 \label{eq: defects}
 \mathbb{P}(A^d_{j,\sigma}(n))\lesssim_d (1+\log n)^d \mathbb{P}(A_{j,\sigma}(n)).
 \end{equation}
 $A_{j,\sigma}(n)$ is the event that there are $j$ arms (without defects) to $\partial S(n)$ (with occupation status as in  $\sigma$). Inspection of the proof in \cite{nolin} reveals that the constant implicit in (\ref{eq: defects}) is of the form $(\C_{29})^d$. Separating the four arms as in \cite{kesten3} verifies that (\ref{eq: defectbound}) holds.
 
 It now remains to show that $\mathbb{P}(A_e) \lesssim \mathbb{P}(B_e)$; that is, that for $n$ sufficiently large, we can change the $p_n(1)$-open arms in the definition of $A_e$ to be $p_c$-open at the cost of a constant probability factor. For edges $e$ at distance $t/2$ from the boundary, this follows immediately from \cite[Lemma 6.3]{damronetal}. We briefly sketch how the proof given there can be adapted to the case where $e$ is close to the boundary. We write $\mathbb{P}(A_e)$ as $\mathbb{P}(A_e(p_n(1),p_c))$, where for $p, q\in [p_c,1)$, $A_e(p,q)$ denotes the event that $e$ has two disjoint $p$-open arms to opposite vertical sides of $J(\mathbf{v})$ and $e^*$ has two disjoint $q$-closed dual arms to the top and bottom of $J(\mathbf{v})$. Using Russo's formula as in \cite[(39)]{damronetal} , we find
\begin{equation}\label{eq: damronetalsum}
 \frac{\mathrm{d}}{\mathrm{d}p}\mathbb{P}(A_e(p,p_c)) =\sum_{e' \neq e}\mathbb{P}(A_e(\cdot,p_c), A_{e'}(\cdot,p),D_{e,e'}(p)).
 \end{equation}
$A_e(\cdot,p_c)$ is the event that $e^*$ has two disjoint $p_c$-closed dual connections to the top and bottom of $J(\mathbf{v})$, and $D_{e,e'}(p)$ is the event that there exist three disjoint $p$-open paths joining, respectively, one vertical side of $J(\mathbf{v})$ to one endpoint of $e$, the other endpoint of $e$ to an endpoint of $e'$, and the other endpoint of $e'$ to the remaining vertical side of $J(\mathbf{v})$. Note that our notation differs somewhat from the one in \cite{damronetal}. For the purposes of illustration, we will henceforth suppose that $e=\langle (v_1+l,v_2+\lfloor 3t/2 \rfloor ),(v_1+l+1,v_2+\lfloor 3t/2 \rfloor ) \rangle$ for some $l < t/4$; that is, $e$ is close to the left side of $J(\mathbf{v})$. The sum on the right of (\ref{eq: damronetalsum}) 
can be rewritten as:
\begin{equation}\label{eq: triplesum}
\left(\sum_{j=1}^{\lfloor l/2 \rfloor} + \sum_{j=\lfloor l/2 \rfloor+1 }^{l} + \sum_{j=l+1}^{3t}\right)\sum_{e':|e'_x-e_x|=j}\mathbb{P}(A_e(\cdot,p_c), A_{e'}(\cdot,p),D_{e,e'}(p)).
\end{equation}
$e_x$ denotes the left endpoint of the edge $e$, if $e$ is a horizontal edge, and its bottom endpoint if $e$ is a vertical edge. The first sum is bounded by
\begin{align}\label{eq: fourfactors} \sum_{j=1}^{\lfloor l/2 \rfloor} \sum_{e':|e'_x-e_x|=j}\mathbb{P}(A_e(\lfloor j/2 \rfloor;p,p_c))\mathbb{P}(A_{e'}(\lfloor j/2 \rfloor;p,p))\\ \times \mathbb{P}(A(\lfloor 3j/2\rfloor,l;p,p_c))\mathbb{P}(T(l,t;p,p_c)) \nonumber \\
\lesssim \sum_{j=1}^{\lfloor l/2 \rfloor}j \mathbb{P}(A_e(\lfloor j/2 \rfloor;p,p_c))\mathbb{P}(A(1,\lfloor j/2 \rfloor;p,p)) \nonumber \\ \times\mathbb{P}(A(\lfloor 3j/2 \rfloor,l;p,p_c))\mathbb{P}(T(l,t;p,p_c)) \nonumber
\end{align}
$A(n,N;p,p_c)$ denotes the probability that there are four arms of alternating occupation status joining $\partial S(n)$ to $\partial S(N)$; $T(l,t;p,p_c)$ is the event that there are two $p_c$-closed arms, as well as a $p$-open arm connecting $\partial S(l)$ to $\partial S(t)$.  Using gluing constructions similar to those in proofs of quasi-multiplicativity, and the fact that we may change the length of any connections involved by constant factors at the cost of constant factors in the probabilities, we have:
\[
\mathbb{P}(A_e(\lfloor j/2 \rfloor;p,p_c))\mathbb{P}(A(\lfloor 3j/2 \rfloor,l;p,p_c))\mathbb{P}(T(l,t;p,p_c)) \asymp \mathbb{P}(A_e(p,p_c)).\]
For $p\le p_n(1)< p_t(1)$, we can use \cite[Theorem~27]{nolin} to assert
\[\mathbb{P}(A(1,\lfloor j/2\rfloor;p,p)) \asymp \mathbb{P}(A(1,\lfloor j/2\rfloor;p_c,p_c)).\]
We can now follow \cite{damronetal} exactly (see equations (42) and (43) and the surrounding discussion) to show 
that the sum in (\ref{eq: fourfactors}) is bounded by:
\[\mathbb{P}(A_e(p,p_c))\cdot l^2 \pi_4(t) \le \mathbb{P}(A_e(p,p_c))\cdot t^2 \pi_4(t).\]
To deal with the second sum in (\ref{eq: triplesum}), we note that when 
\[|e_x-e'_x|\ge \lfloor l/2 \rfloor+1,\] the conjunction of the events $A_e(\cdot,p_c), A_{e'}(\cdot,p)$ and $D_{e,e'}(p)$ appearing in the probability on the right of the equation implies that $e$ has 2 $p$-open, and $e^*$ two $p_c$-closed arms to distance $\lfloor l/4 \rfloor$, that $e'$ has four alternating arms with parameter $p$ to the boundary of the intersection of $S(e'_x,\lfloor l/4 \rfloor)$ with $J(\mathbf{v})$, three of which reach to distance $\lfloor l/4 \rfloor$, and finally that $\partial S(e_x,\lfloor 5l/4 \rfloor)$ has two $p_c$-closed arms to the top and bottom of $J(\mathbf{v})$ and a $p$-open arm to the right side of $J(\mathbf{v})$, and all these connections occur inside $J(\mathbf{v})$. Using these observations, an argument similar to the previous case and a summation analogous to that in the proof of \cite[Lemma 6.2]{werner}, shows that we can estimate (in addition, using the remarks in \cite[Section~4.6]{nolin} to change the $p$-open and closed arms in a half-plane to $p_c$-open and closed arms)
\[\sum_{j=\lfloor l/2 \rfloor+1 }^{l} \sum_{e':|e'_x-e_x|=j}\mathbb{P}(A_e(\cdot,p_c), A_{e'}(\cdot,p),D_{e,e'}(p)) \lesssim \mathbb{P}(A_e(p,p_c)) \cdot l^2\pi_4(l).\]

\begin{figure}[h]
\centering
\includegraphics[scale=0.5]{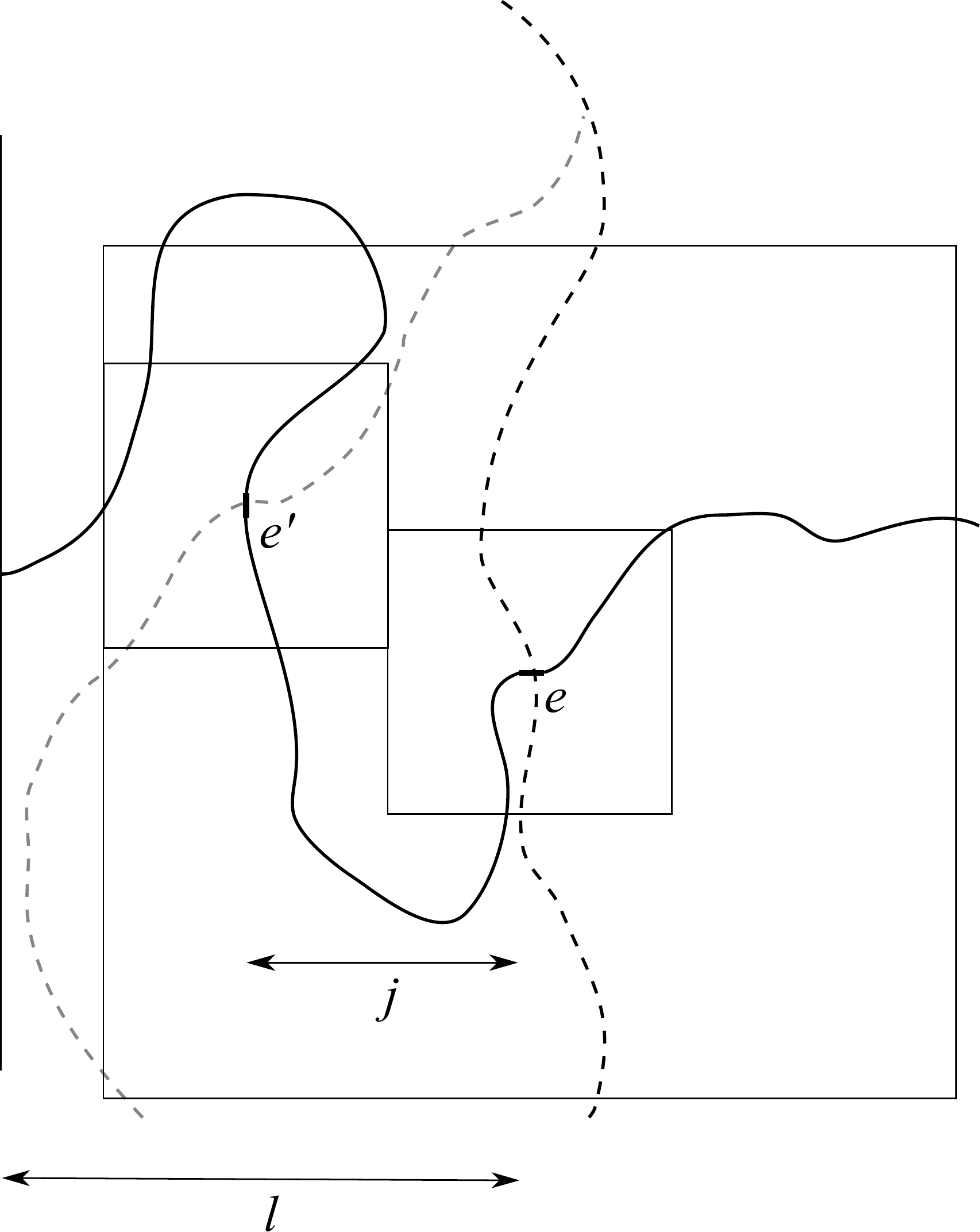}
\caption{Illustration of the setup for management of the second sum in \eqref{eq: triplesum}. The edge $e$ is at distance $l$ from the left boundary of $J(\mathbf{v})$ and the distance between $e'$ and $e$ is $j$, a number between $l/2$ and $l$. The dark dotted curve represents a $p_c$-closed dual path (given by Menger's theorem) and the dark solid curve represents a $p_n(1)$-open path, connecting $e'$ and $e$ to each other and to the left and right sides of $J(\mathbf{v})$. The grey dotted curve represents a $p_n(1)$-closed dual path connecting the edge dual to $e'$ with the top and bottom of $J(\mathbf{v})$.}
\label{fig: boxes}
\end{figure}

Turning to the final sum on the left in (\ref{eq: triplesum}), we can again closely follow \cite{damronetal} to bound this term by
\[t^2\pi_4(t)\cdot \mathbb{P}(A_e(p,p_c)).\]
The estimates outlined above for the left side of (\ref{eq: damronetalsum}) imply
\[\frac{\mathrm{d}}{\mathrm{d}p}\log \mathbb{P}(A_e(p,p_c)) \le C(l^2\pi_4(l)+t^2\pi_4(t)).\]
Integrating this from $p_c$ to $p_n(1)$ and using (\ref{eq: nearcritical}), we find
\[\mathbb{P}(A_e(p_n(1),p_c))\lesssim \mathbb{P}(A_e(p_c,p_c)),\]
which is what we wanted to prove.
We have thus established (\ref{eq: ecomp}); that is, we have shown 
 \[\mathbb{P}(\Xi, K=k)\le \mathbb{E}[\sharp N_k] \lesssim (\C_{30}\log t)^{2k}\cdot(p_n(1)-p_c) \sum_{e\in J(\mathbf{v})}\mathbb{P}(B_e),\]
where $N_k$ is the number of edges satisfying the conditions in Lemma \ref{crossing-lemma}. Note that $B_e$ is equal to the event that the edge $e$ is \emph{pivotal} for the existence of a left-right $p_c$-open crossing of $J(\mathbf{v})$. Following \cite[Lemma 6.2]{werner}, we can show
 \[ \sum_{e\in J(\mathbf{v})}\mathbb{P}(B_e) \lesssim t^2\cdot \pi_4(t).\]
This concludes the proof of the lemma.

\subsection{Final Peierls argument}

We use a block argument and a Peierls argument to upgrade (\ref{eq:pre-peierls-bound}) to (\ref{eq-crossing-bound}). The annulus $F(\mathbf{j},q)\setminus D(\mathbf{j},q)$, centred at
\[v= q\left(j_1+\frac{1}{2}, j_2+\frac{1}{2}\right)\]
is tiled with smaller squares of side length
\[t = \frac{q}{\log q}.\]
The existence of a $p_n(1)$-open crossing of $F(\mathbf{j},q)\setminus D(\mathbf{j},q)$ implies that of a crossing 
\[\bar{r}= (x(0), x(1), \ldots, x(\xi)),\] 
of $S(v,3q/2-5t)\setminus S(v,q/2+5t)$ along edges of $\mathbb{Z}^2$, with $x(0) \in S(v,q/2+5t)$ and $x(\xi)\in S(v,3q/2) \setminus (S(v,3q/2-5t))^\circ$.
The reason for considering this smaller annulus will become clear below. We can now introduce sequences $\mathbf{j}_0,\ldots, \mathbf{j}_\lambda$, and $l_0=0, \ldots, l_\lambda$ relative to the sequence $x(i)$ and squares of size $t$; that is,
\[x(l_i) \in D(\mathbf{j}_i,t)\]
\[l_{i+1} = \min\{l>l_i:x(l)\notin F(\mathbf{j}_i,t)\}.\]

The first observation is that we have a lower bound on $\lambda$ due to the difference in scales:
\[q-10t \le |x(\xi)-x(0)|\le \sum_{l=0}^{\lambda-1}|x(l_{i+1})-x(l_i)|+|x(\xi)-x(l_\lambda)|\le 2\sqrt{2}t(\lambda +1),\]
implying
\[\lambda \geq \C_{31} \log q.\]
The second observation is that
\[|\mathbf{j}_{i+1}(k)-\mathbf{j}_i(k)|\le 2, \  k=1,2,\]
where $\mathbf{j}_i(k)$ denotes the $k$-th coordinate of the vector $\mathbf{j}$.
From this, for each fixed $\lambda$, given $\mathbf{j}_i$, there are at most 16 choices for $\mathbf{j}_{i+1}$ and so at most
\[4\left(\frac{q}{t}+11\right)16^\lambda\]
choices for the sequence $\mathbf{j}_0, \ldots, \mathbf{j}_\lambda$. The first factor is an estimate for the number of choices of squares $D(\mathbf{j},t)$ with $x(0)\in D(\mathbf{j},t)$.  

The third observation is that $\bar{r}$ must contain, between $x(\mathbf{j}_i)$ and $x(\mathbf{j}_{i+1})$, a ``short'' crossing $r_i$ of a $t\times 3t$ or $3t\times t$ rectangle $R_i$ (that is, the crossing is between the long sides). 

Denote by $\tilde{R}_i$ the $2t\times 5t$ or $5t\times 2t$ rectangle around $R_i$, as in (\ref{eq:pre-peierls-bound}). Then 
\[\tilde{R}_i \subset S(v,3q/2)\setminus S(v,q/2),\]
and so
\begin{multline*}\hat{Z}(v,q/2,\bar r) = \{x\in S(v,3q/2): x\xrightarrow{p_c} \bar r \text{ in } S(v,3q/2)\setminus S(v,q/2)\}\\
\ge \max_{0\le i \le \lambda} Z(\tilde{R}_i,r_i),
\end{multline*}
where $Z(\tilde{R}_i,r_i)$ is the number of points in $\tilde{R}_i$ connected to $r_i$ by a $p_c$-open path in $\tilde{R}_i$. It follows that
\begin{multline}
\mathbb{P}\left(\begin{gathered} \exists p_{2m}(1)\text{-open crossing of } F(\mathbf{j},q)\setminus D(\mathbf{j},q) \\ \text{ with } \hat{Z}(v,q/2,r) \le q^{2}\pi(q)/(\log q)^4\end{gathered}\right)\\
\le \sum_{\tilde{R}_0, \ldots, \tilde{R}_\lambda}\mathbb{P}\left(\begin{gathered} \text{for all } i \le \lambda, \exists p_{2m}(1)\text{-open crossing } r_i \text{ in } R_i\\ \text{ with } Z(\tilde{R}_i,r_i)\le q^{2}\pi(q)/(\log q)^4\end{gathered}\right). \label{eq: peierls-sum}
\end{multline}
The sum is over all possible finite sequences of squares $\{\tilde{R}_i\}_{i\le \lambda}$, for all $\lambda \ge \C_{31}\log q$. This quantity is controlled by choosing a subsequence of $\C_{32}\lambda$ disjoint $\tilde{R}_i$: each rectangle intersects a fixed number of other such rectangles. The events appearing in the last probability are independent for disjoint $\tilde{R}_i$'s. Their probability can be bounded using (\ref{eq:pre-peierls-bound}) (with $\delta=1$), since our choice of $t$ implies
\[\frac{q^{2}}{(\log q)^4}\pi(q) \leq \C_{33}\frac{t^2}{\log t}\pi(t),\]
for large $q$.
Moreover, one can use the bound on the number of sequences of $\mathbf{j}$'s (there are at most 4 choices of $R_i$ for a given $\mathbf{j}_i$) to control the entire sum: the last line in (\ref{eq: peierls-sum}) is bounded up to a constant factor by:
\[\sum_{\lambda \ge \C_{31}\log q}q64^{\lambda}(\C_{33}(\log t)^{-\C_{34}})^{\C_{32}\lambda}.\]
For $q$ large enough, this sum is bounded (again up to a constant) by:
\[\exp(-\C_{35}\log q \cdot \log \log t) \ll q^{-c}\]
for any $c>0$.

On $H_{2m}(1)$, any crossing $r$ in the portion of the IPC $\Gamma(n)$ consists of $p_{2m}(1)$-open edges. Since any site $p_c$-connected to a site in the IPC also belongs to the IPC, we find that the probability in (\ref{eq-crossing-bound}) is bounded by:
\begin{multline}
\mathbb{P}(H_{2m}(1)^c) + \mathbb{P}\left(\begin{gathered} \exists \text{ a } p_{2m}(1) \text{-open crossing of } F(\mathbf{j},q)\setminus D(\mathbf{j},q) \\ \text{ with } \hat{Z}(v,q/2,r) \le q^{2}\pi(q)/(\log q)^4\end{gathered}\right)\\ \lesssim (2m)^{-M\C_0}
+ \exp(-\C_{35}\log q(\log\log q -\log\log\log q)).
\end{multline}
Choosing $M$ appropriately in the definition of $p_n(1)$ (depending on the parameter $c$ in (\ref{eq-crossing-bound})) establishes the claim.

\appendix
\section{Quenched subdiffusivity on the Incipient Infinite Cluster}
In this section, we justify Remark~3 above and outline the derivation of a result analogous to Theorem \ref{kesten-ipc-1} for the random walk on H. Kesten's \emph{Incipient Infinite Cluster} (IIC). For cylinder events $A$, the IIC measure is defined by
\begin{equation} \label{eq:iic-def} \mathbb{P}_{\mathrm{IIC}}(A) =\lim_{l\rightarrow \infty} \mathbb{P}_{p_c}(A\mid \mathbf{0} \rightarrow \partial S(l)).\end{equation}
It was shown in \cite{kesten2} that the limit (\ref{eq:iic-def}) exists and that the resulting set function extends to a measure. Note that the connected cluster of the origin,  $C(\mathbf{0})$, is $\mathbb{P}_{\mathrm{IIC}}$-almost surely unbounded. We will refer to this cluster as the IIC. We have the following result:
\begin{theorem}[Quenched Kesten theorem for the IIC]
Let $\{X_k\}_{k\ge 0}$ denote a simple random walk on the incipient infinite cluster started at $\mathbf{0}$. Let $\tau(n)$ denote the first exit time of $X_k$ from $S(n)$. There exists $\epsilon > 0$ such that, for $\mathbb{P}_{\mathrm{IIC}}$-almost every $\omega$ and almost-every realization of $\{X_k\}$, there is a (random) $n_0$  such that
\[\tau(n)\ge n^{2+\epsilon}\]
for $n$ greater than $n_0$.
\end{theorem}
We can proceed along the lines of the proof of estimate (\ref{eq: W1bound}), and consider a suitable modification of the random walk whose distribution coincides with that of $X$ from the first hitting time $\tau(2m)$ of $\partial S(2m)$ to the first hitting time of $\partial S(m) \cup \partial S(3m)$ after time $\tau(2m)$, $\sigma^+(m)$. To use the argument leading to (\ref{eq: W1bound}) in our case, we merely need to show that we can prove an estimate equivalent to the one obtained for $\mathbb{P}_{\mathrm{IPC}}(E_1(n)^c)$ in Section~\ref{section-comparison}.

We will show that there are constants $C>0$ and $s>1$ such that
\begin{equation}\label{eq: piicbound}
\mathbb{P}_{\mathrm{IIC}}(\{\operatorname{dist}_{C(\mathbf{0})}(\partial S(2m),\partial S(n)\cup \partial S(m)) \le C n^{s}\})\lesssim n^{-2},
\end{equation}
By the argument given in the proof of Lemma \ref{pisztoraslemma}, there exists $C>0$ and $s>1$ such that
\begin{equation} \label{eq: pcbound} \mathbb{P}_{p_c}\left(\begin{gathered} \exists \text{ an open path in } S(3m)\setminus S(m)^\circ \text{ connecting } \partial S(2m)\\ \text{ to } \partial S(m) \text{ or } \partial S(3m) \text{ with less than } Cn^s \text{ edges}\end{gathered}\right)\lesssim n^{-2}.\end{equation}
Let us denote the event on the left by $G(n)$. Clearly
\[\{\operatorname{dist}_{C(\mathbf{0})}(\partial S(2m),\partial S(n)\cup \partial S(m)) \le C n^{s}\} \subset G(n).\]
$G(n)$ depends only on the status of edges inside $S(3m)\setminus S(m)^\circ$.
Write the conditional probability in the definition of $\mathbb{P}_{\mathrm{IIC}}$ as a ratio:
\begin{align*}
\mathbb{P}_{p_c}(G(n)\mid \mathbf{0} \rightarrow \partial S(l)) &= \frac{\mathbb{P}_{p_c}(G(n), \mathbf{0} \rightarrow \partial S(l))}{\mathbb{P}_{p_c}(\mathbf{0} \rightarrow \partial S(l))}.
\end{align*}
For $l>3m$, we have, by independence and monotonicity
\begin{equation}\label{eq: iic3prod}
\mathbb{P}_{p_c}(G(n), \mathbf{0} \rightarrow \partial S(l)) \le \mathbb{P}_{p_c}(G(n))\mathbb{P}_{p_c}(\mathbf{0}\rightarrow \partial S(m))\mathbb{P}_{p_c}(\partial S(3m) \rightarrow \partial S(l)).
\end{equation}
Now
\[\mathbb{P}_{p_c}(\mathbf{0}\rightarrow \partial S(m)) \asymp  \mathbb{P}_{p_c}(\mathbf{0}\rightarrow \partial S(3m)),\]
and by quasi-multiplicativity
\[\mathbb{P}_{p_c}(\mathbf{0}\rightarrow \partial S(3m))\cdot\mathbb{P}_{p_c}(\partial S(3m) \rightarrow \partial S(l)) \asymp   \mathbb{P}_{p_c}(\mathbf{0} \rightarrow \partial S(l)).\]
Using this in (\ref{eq: iic3prod}), we have, by (\ref{eq: pcbound}):
\[\mathbb{P}_{p_c}(G(n)\mid \mathbf{0} \rightarrow \partial S(l)) \lesssim n^{-2},\]
from which (\ref{eq: piicbound}) follows at once.

\bigskip
\noindent
{\bf Acknowledgements.} We thank T. Kumagai for suggesting the problem of proving a quenched analogue of Kesten's subdiffusivity theorem and for comments on a previous verion. We are very grateful to A. Fribergh for comments that led to a substantial reorganization of the presentation. J. H. and P. S. thank M. Aizenman for advising and thank the organizers of the workshop ``Current topics in Mathematical Physics'' at the Erwin Schr\"odinger Institute, where some of this work was done.

\


\begin{thebibliography}{9}
\bibitem[AB]{aizburch} Aizenman, M., Burchard, A., \emph{H\"older regularity and dimension bounds for random curves}, Duke Math. J., \textbf{99}, 3, 1999.

\bibitem[AGHS]{AGHS} Angel, O., Goodman, J., den Hollander, F., Slade, G., \emph{Invasion percolation on regular trees}, Ann. Probab., \textbf{36}, 2008.

\bibitem[BJKS]{BJKS} Barlow, M., J\'arai, A. A., Kumagai, T., Slade, G., \emph{Random walk on the incipient infinite cluster for oriented percolation in high dimensions}, Comm. Math. Phys., \textbf{278}, 2008.

\bibitem[BK]{BK} Barlow, M., Kumagai, T., \emph{Random walk on the incipient infinite cluster on trees}, Illinois J. Math., \textbf{50}, 2006.

\bibitem[BN]{BN11} Beffara, V., Nolin, P., \emph{On monochromatic arm exponents for 2D critical percolation}, Ann. Probab., \textbf{39}, 4, 2011.

\bibitem[vdBK]{vdbkesten} van den Berg, J., Kesten, H., \emph{Inequalities with applications to percolation theory and reliability}, J. Appl. Prob. \textbf{22}, 1985.

\bibitem[C]{carne} Carne, Th. K., \emph{A transmutation formula for Markov chains}, Bull. Sci. Math. \textbf{2} 109, 1985.

\bibitem[CCN]{2chayesnewman} Chayes, J.T., Chayes, L., Newman, C. \emph{The stochastic geometry of invasion percolation.}, Comm. Math. Phys., \textbf{101}, 1985.

\bibitem[CK]{CK} Croydon, D., Kumagai, T. \emph{Random walks on Galton-Watson trees with infinite variance offspring distribution conditioned to survive}, Electron. J. Probab., \textbf{13}, 2010.

\bibitem[D]{diestel} Diestel, R., \emph{Graph Theory}, Fourth Edition, Graduate Texts in Mathematics 173, Springer, 2010.

\bibitem[DSV]{damronetal} Damron, M., Sapozhnikov, A., V\'av\"olgyi, B., \emph{Relations between invasion percolation and critical percolation in two dimensions}, Ann. Prob., \textbf{37}, 2009.

\bibitem[GPS]{garbanpeteschramm} Garban, C., Pete, G., Schramm, O. \emph{The Fourier spectrum of critical percolation}, Acta Mathematica, \textbf{1}, 2010.

\bibitem[K1]{kesten} Kesten, H., \emph{Subdiffusive behavior of random walk on a random cluster}, Ann. Inst. H. Poincare, \textbf{22}, 4, 1986.

\bibitem[K2]{kesten2} Kesten, H., \emph{The incipient infinite cluster in Two-Dimensional Percolation}, Probab. Th. Rel. Fields, \textbf{73}, 1986. 

\bibitem[K3]{kesten3} Kesten, H., \emph{Scaling relations in 2D percolation}, Comm. Math. Phys., \textbf{109}, 1, 1987.

\bibitem[K4]{kesten4} Kesten, H. \emph{The critical probability of bond percolation on the square lattice equals } $1/2$, Comm. Math. Phys., \textbf{74}, 1, 1980.

\bibitem[K5]{kestenbook} Kesten, H. \emph{Percolation theory for mathematicians}, Birkh\"auser, 1982.

\bibitem[KN]{KN} Kozma, G., Nachmias. A. \emph{The Alexander-Orbach conjecture holds in high dimensions}, Invent. Math., \textbf{178}, 2009.

\bibitem[J]{jarai} J\'arai, A., \emph{Invasion percolation and the incipient infinite cluster in 2D.}, Comm. Math. Phys. \textbf{236}, 2003.

\bibitem[Ng]{nguyen} Nguyen, B. G., \emph{Typical cluster size for 2-dim percolation processes (Revised)}, IMA Preprint Series 324, 1987. 

\bibitem[No]{nolin} Nolin, P., \emph{Near-critical percolation in two dimensions}, Electron. J. Probab. \textbf{13}, 2008.

\bibitem[LP]{pereslyons} Lyons, R., Peres, Y., \emph{Probability on Trees and Networks}, \url{http://mypage.iu.edu/~rdlyons/pbtree/pbtree.html}. Accessed February 2012.

\bibitem[Pisz]{pisztora} Pisztora, A., \emph{Scaling inequalities for shortest paths in regular and invasion percolation}, Carnegie-Mellon CNA preprint, available at \url{http://www.math.cmu.edu/CNA/Publications/publications2000/001abs/00-CNA-001.pdf}

\bibitem[R]{reimer} Reimer, D., \emph{Proof of the van den Berg-Kesten conjecture.} Combin. Probab. Comput. \textbf{9}, 27-32.

\bibitem[S12]{S12} Shiraishi, D., \emph{Random walk on non-intersecting two-sided random walk trace is subdiffusive in low dimensions}, To appear in Trans. Amer. Math. Soc.

\bibitem[V]{varopoulos} Varopoulos, N., \emph{Long range estimates for Markov chains.} Bull. Sci. Math. \textbf{2}, 109, 1985.

\bibitem[W]{werner} Werner, W. \emph{Lectures on two-dimensional critical percolation}, IAS-Park City Summer School notes, July 2009.

\bibitem[Z]{zhang} Zhang, Y. \emph{The fractal volume of the two-dimensional invasion percolation cluster.} Comm. Math. Phys., \textbf{167}, (1995), 237-254.
\end{thebibliography}
\end{document}